\documentclass[11pt, preprint]{amsart}
\usepackage{amsfonts,amssymb,amsmath,amsthm,amscd,mathtools,multicol,tikz, tikz-cd,caption,enumerate,mathrsfs,geometry,tgpagella}
\usepackage[foot]{amsaddr}
\usepackage[pagebackref=true, colorlinks, linkcolor=blue, citecolor=red]{hyperref}
\usepackage{fullpage}
\usepackage[all]{xy}

\makeatletter
\def\blfootnote{\gdef\@thefnmark{}\@footnotetext}
\makeatother

\makeindex

\newtheorem{theorem}{Theorem}[section]
\newtheorem{proposition}[theorem]{Proposition}
\newtheorem*{theorem*}{Theorem}

\newtheorem{lemma}[theorem]{Lemma}
\theoremstyle{remark}
\newtheorem{remark}[theorem]{Remark}
\newtheorem{example}[theorem]{Example} 

\numberwithin{equation}{subsection}

\begin{document}
\title{Strongly normal extensions and Algebraic Differential Equations}
\author{Partha Kumbhakar and Varadharaj Ravi Srinivasan}
\address{Indian Institute of Science Education and Research (IISER) Mohali Sector 81, S.A.S. Nagar, Knowledge City, Punjab 140306, India.}
\email{\url{parthakumbhakar@iisermohali.ac.in}, \url{ravisri@iisermohali.ac.in}}	
\thanks{The first named author would like to acknowledge the support of DST-FIST grant SR/FST/MS-I/2019/46(C)}
\setcounter{tocdepth}{1}

\begin{abstract}
Let $k$ be a differential field having an algebraically closed field of constants $C$, $E$ be a strongly normal extension of $k$, and $k^0$ be the algebraic closure of $k$ in $E.$  We prove for any intermediate differential field $k\subset K\subseteq E$ that there is an intermediate differential field $k\subset M\subseteq K$ such that either  $M$ is generated  as a differential field over $k$ by a nonalgebraic solution of a Riccati differential equation over $k$ or $k^0M$ is an abelian extension of $k^0$. Using this result,  we reprove and extend the works of Goldman  (\cite[Theorem 3 \& 4]{GL59} and Singer \cite[Corollary 3 \& 4]{Sin76}) and obtain that if a solution of an irreducible linear homogeneous differential equation $\mathcal{L}(t)=0$ satisfies an algebraic differential equation of order $d\leq 4,$ then every solution of $\mathcal{L}(t)=0$ can be expressed in terms of algebraic elements and solutions of linear homogeneous differential equations of order $\leq d.$  We also extend a work of Rosenlicht \cite[Proposition]{Ros74} and prove the following results on algebraic dependency of solutions of algebraic differential equations: a) Let $n\leq 4$ be a positive integer,  $y_1, y_2,\dots, y_{n+3}$ be distinct generic solutions  of an algebraic differential equation over $k$ of order $n$ such that the differential field $k\langle y_1,\dots, y_{n+3}\rangle$ has $C$ as its field of constants. Then, the field transcendence degree of $k\langle y_1,\dots, y_{n+3}\rangle$ is at most $n(n+3)-1$ over $k$  if and only if $k\langle y_1\rangle$ has a nonalgebraic differential subfield that can be embedded in a strongly normal extension of $k.$ b) Let $n$ be any positive integer and $y_1, y_2$ be distinct generic solutions  of an algebraic (autonomous) differential equation over $C$ of order $n$ such that  $C\langle y_1,y_2\rangle$ has $C$ as its field of constants. Then, the field transcendence degree of $C\langle y_1,y_2\rangle$ is at most $2n-1$ over $C$ if and only if $C\langle y_1\rangle$ contains a nonalgebraic strongly normal extension of $C.$
\end{abstract}
\maketitle
\tableofcontents

\section{Introduction}\label{introduction}

Let $k$ be a differential field of characteristic zero with an algebraically closed field of constants $C$.   In this article, using the theory of algebraic groups, we prove several theorems describing the structure of differential subfields of strongly normal extensions of $k.$ These theorems are then used to study certain algebraic dependence of solutions of nonlinear differential equations.

Let $E$ be a strongly normal extension of $k$ and let $K$ be a differential subfield of $E$ containing $k.$ In the next few paragraphs, we record all the results available in the literature describing the structure of $K.$ If $\mathrm{tr.deg}(K|k)\footnote{Field transcendence degree of $K$ over $k.$}=1$ then it is known that for some finite algebraic extension $\tilde{k}$ of $k,$ $\tilde{k}K=\tilde{k}\langle y\rangle,$ where either $y$ is a solution of a Riccati differential equation $y'=ay^2+by+c$ with $a,b,c\in \tilde{k}$ not all zero or $y$ is a solution of a Weierstrass differential equation $y'^2=\alpha^2(4y^3-g_2y-g_3)$ with $\alpha\in \tilde{k}\setminus \{0\},$ $g_2,g_3\in C$ and $27g^2_3-g^3_2\neq 0.$ This result is first proved by  Matsuda \cite{Mat80} and Nishioka \cite{NK89} using the theory of differential field depends rationally on arbitrary constants and its connection with differential algebraic function field in one variable having no movable singularity. One can also prove the result using differential Galois theory (see \cite[Theorem 1.1]{KRS24}).

If $k=C$ and $E$ is a Picard-Vessiot extension, then the differential Galois group of $E$ over $C$ is a commutative algebraic group having one of the following forms: $\prod^r \mathrm G_m$,  $\mathrm G_a$ or $\mathrm G_a \times \prod^r \mathrm G_m,$ for some integer $r\geq 1.$ From this it follows that $E=C(x_1,x_2,\dots,x_n),$ where $x'_i/x_i\in C$ for all $i=2,\dots, n$ and either $x'_1=1$ or $x'_1/x_1\in C$. If $K$ is a differential subfield of $E$  then since the differential Galois group is commutative, $K$ is also a Picard-Vessiot extension of $C$ and  $K=C(y_1,y_2,\dots,y_l),$ where $y'_i/y_i\in C$ for all $i=2,\dots, l$ and either $y'_1=1$ or $y'_1/y_1\in C.$ Thus, in this case, we can completely describe the structure of any differential subfield of $E.$

Let $E$ be a Picard-Vessiot extension of $k.$ Then, from a result of Yves Andr\'{e} \cite{And14}, we know that $K$ is a solution field, that is, $K$ is generated as a field over $k$ by solutions of a linear differential equation over $k$ if and only if the corresponding closed subgroup of the differential Galois group is an observable subgroup.  Therefore, the structure of certain differential subfields of a Picard-Vessiot extension can be well described.  However, if the differential Galois group is a nonsolvable linear algebraic group of positive dimension, then it contains nonobservable subgroups, such as the Borel subgroups or the parabolic subgroups, which in turn provide us, through Galois correspondence, with differential subfields that may not contain any nonzero solutions of linear homogeneous differential equations. 

Thus, in general, if $k\neq C$ and $K$ is a differential subfield of a Picard-Vessiot extension, or in general a strongly normal extension of $k$ with $\mathrm{tr.deg}(K|k)\geq 2$  then there are no complete results available in the literature that describe the structure of $K.$  

\subsection{Main results} \sloppy A finitely generated differential field extension $K$ of $k$ having $\mathrm{tr.deg}(K|k)\geq 1$ is said to be \emph{irreducible} over $k$ or $k-$\emph{irreducible} if for every differential  field $k\subset M\subseteq K,$  either $M$ is an algebraic extension of $k$ or  $K$ is an algebraic extension of $M;$ equivalently,  either $\mathrm{tr.deg}(M|k)=0$ or $\mathrm{tr.deg}(M|k)=\mathrm{tr.deg}(K|k).$ It is easily seen that any finitely generated differential field extension $K$ of $k$ is either a finite algebraic extension of $k$ or it can be resolved into a tower of differential fields $$k=K_0\subseteq K_1\subseteq K_2\subseteq \cdots \subseteq K_{m-1}\subseteq K_m=K$$ such that for each $i=1,\dots, m,$ $K_i$ is a $K_{i-1}-$irreducible differential field.  

\begin{theorem}\label{intromain}
	Let $k\subset K\subseteq E$ be differential fields,  $E$ be a Picard-Vessiot extension of $k$ and $K$ be  $k-$irreducible. Then, there are a nonzero element $y\in E$ and a monic linear differential operator $\mathcal L\in k[\partial]$ such that $\mathcal L(y)=0$ and  $K$ is a finite algebraic extension of $k\langle y'/y\rangle.$ 
\end{theorem}


\begin{theorem}\label{intromain-SN} Let  $k\subset K\subseteq E$ be differential fields,  $E$ be a strongly normal extension of $k$ and $k^0$ be the algebraic closure of $k$ in $E.$ 
	\begin{enumerate}[(a)]
	\item \label{intromain-irred} If $K$ is  $k-$irreducible then either  $k^0K$ is  an abelian extension of  $k^0$ or $k\subseteq K\subseteq L\subseteq E,$ where $L$ is a Picard-Vessiot extension of $k.$\\
    
	\item \label{intromain-uni} If the field $K$ can be embedded in a purely transcendental extension field of $k,$ then $k\subseteq K\subseteq L\subseteq E,$ where $L$ is a Picard-Vessiot extension of $k.$\\
	\end{enumerate}
	\end{theorem}
		
A differential field extension $E$ of $k$ is said to an \emph{iterated strongly normal extension} (respectively \emph{iterated Picard-Vessiot extension}) of $k$ if there is a tower of differential fields $$k=E_0\subseteq E_1\subseteq \cdots\subseteq E_{n-1}\subseteq E_n=E$$ such that for each $i=1,2,\dots,n,$ $E_i$ is a strongly normal extension (respectively Picard-Vessiot extension) of $E_{i-1}.$

\begin{theorem}\label{iter-SN-Subfields}\label{subfields-IPV}
Let  $k\subset K\subseteq E$ be differential fields and  $E$ be an iterated strongly normal extension (respectively an iterated Picard-Vessiot extension) of $k.$ Then, $K$ can be resolved into a tower of differential fields $$k=K_0\subset K_1\subset K_2\subset \cdots \subset K_{m-1}\subset K_m=K$$ such that for each $i=1,\dots, m,$ $K_i$ is a differential subfield of a strongly normal extension \footnote{cf. \cite[Theorem]{NK94} and \cite[Theorem 23]{HU90}.} (respectively a Picard-Vessiot extension) of $K_{i-1}.$ 
\end{theorem}

We shall now make a few remarks on our results. Let $\mathcal L(y)=0$ be linear homogeneous differential equation of order $n$ over $k$ and $\mathcal R_\mathcal L$ be the Riccati differential equation\footnote{If $\mathcal L=\partial^2+a_1\partial+a_0$ then $\mathcal R_\mathcal L(w,w')=w'+w^2+a_1w+a_0$ and if $\mathcal L=\partial^3+a_2\partial^2+a_1\partial+a_0$ then $\mathcal R_\mathcal L(w,w',w'')=w''+3ww'+a_2(w'+w^2)+a_1w+a_0.$} associated with $\mathcal L.$ For a nonzero element $\zeta,$  $\mathcal L(\zeta)=0$ if and only if $w=\zeta'/\zeta$ is a solution of Riccati differential equation $\mathcal R_\mathcal L\left(w, w',\dots,w^{(n-1)}\right)=0$ of order $n-1.$  Since differential fields of transcendence degree $1$ are irreducible,  Theorems \ref{intromain}  and \ref{intromain-SN} directly generalizes \cite[Theorem 1.1]{KRS24} to $k-$irreducible differential fields.

 A finitely generated differential field extension $K$ of $k$ is said to have \emph{no movable singularity} if $k$ is algebraically closed in $K$ and there exists a nonsingular projective variety $(\Gamma, \mathscr O_\Gamma)$ over $k$ such that its function field $k(\Gamma)$ is $k-$isomorphic to $K$ and that $\mathscr O_\Gamma(U)$ is a differential subring of $K$ for every open set $U\subset \Gamma.$  It is shown in \cite[Theorem 2.1]{Bui86} that a differential field extension $K$ of an algebraically closed field $k$ with $C_K=C$ has no movable singularity over $k$ if and only if $K$ can be embedded in a  strongly normal extension $E$ of $k$ such that $E$ is a Picard-Vessiot extension of $K.$ Thus, in particular,  differential subfields of a Picard-Vessiot extension have no movable singularities, and strongly normal extensions have no movable singularities. It is also shown in \cite[page 45]{Bui86} that if $K$ is a differential field extension of an algebraically closed differential field $k,$ $C_K=C$ and $\mathrm{tr.deg}(K|k)=1$ then $K$ is generated as a field over $k$ by either a solution of a Riccati differential equation or a solution of a Weierstrass differential equation. From our Theorems \ref{intromain}, \ref{intromain-SN} and \cite[Theorem 2.1]{Bui86}, we can readily deduce the following: \begin{enumerate}[(i)]\item  If $k$ is an algebraically closed field, $K$ is  $k-$irreducible, $C_K=C$ and that $K$ has no movable singularity then either $K$ is an algebraic extension of  $k\langle w\rangle,$ where $w$ is a solution of a Riccati differential equation over $k$ or $K$ is an abelian extension of $k$.\\  \item  If $K$ is a differential subfield of a strongly normal extension of an algebraically closed differential field $k$ and that $K$ is also a unirational field then $K$ has no movable singularity.  \end{enumerate}

The proof of Theorem \ref{iter-SN-Subfields} requires the works of Umemura and Nishioka (\cite{HU90} \& \cite{NK94}) on differential fields depending rationally on arbitrary constants and our Theorem \ref{intromain-SN}(\ref{intromain-irred}). Observe that if $K$ is a $k-$irreducible differential subfield of an iterated strongly normal extension of $k$, then it contains a $k-$irreducible subfield $M$ that can be embedded in a strongly normal extension of $k.$ Now from Theorem \ref{intromain-SN} and Theorem \ref{intromain}, we can further describe the structure of $M.$  As we will see, this observation allows us to study solution fields of nonlinear differential equations.  
 
\subsection{Applications to nonlinear differential equations}
An \emph{algebraic differential equation} over $k$ of order $n$ is a differential equation of the form \begin{equation}\label{aode} f(t,t',\dots,t^{(n)})=0,\end{equation} where $f$ is an irreducible polynomial in $n+1$ variables over $k$ and $\partial f/\partial t^{(n)}\neq 0.$ If the coefficients of $f$ are in $C$ then the differential equation (\ref{aode}) is said to be \emph{autonomous}. Let $\mathfrak U$ be a universal extension of $k$ (\cite{Kol73}).  A \emph{solution} of (\ref{aode}) is an element $y\in \mathfrak U$ such that $f(y,y',\dots,y^{(n)})=0$ and that $k\langle y\rangle$ has $C$ as its field of constants (cf.  \cite{NNvdPT15} \& \cite{NMvdPT22}).

In \cite[Proposition]{Ros74}, Rosenlicht proves the following: Let $k$ be a differential field with constants $C,$ and $k(x)$ be a pure transcendental extension field of $k.$ For a nonzero rational function $f(x)\in C(x),$  extend the derivation on  $k$ to $k(x)$ by defining $x'=f(x).$ Suppose that $1/f(x)$ is not of the form \begin{equation*} \frac{\partial u}{\partial x} \ \ \text{or}\ \frac{1}{cu}\frac{\partial u}{\partial x}\end{equation*} for any $u(x)\in  C(x)$ and nonzero constant $c\in C$. Then, every constant of $k(x)$ is in $C.$

More recently, Eagles and Jimenez \cite[Corollary]{eagles-Jimenez} extended this result of Rosenlicht to systems of autonomous first order differential equations using model-theoretic methods and asked for an extension of their result in the context of algebraic differential equations of arbitrary order (see \cite[Section 5]{eagles-Jimenez}).  Using \cite{NK89a}, \cite{HU90}, and our main results (Theorems \ref{intromain} and \ref{intromain-SN}), we establish an extension of the Rosenlicht’s result (see Proposition \ref{strong-ros-extn} and Remarks \ref{ros-extn}, \ref{EJ-result}) and thus answering the question of Eagles and Jiménez.

Let $y$ be a solution of (\ref{aode}) and let $\mathfrak S_{k,f,y}\subset \mathfrak U$ be the set of all solutions of (\ref{aode}) satisfying the following conditions \begin{enumerate}[(i)] \item For any $y_1\in \mathfrak S_{k,f,y},$  there is a differential $k-$isomorphism between the differential fields $k\langle y\rangle$ and $k\langle y_1\rangle$ that maps $y$ to $y_1.$ \item  For any  $y_1, y_2,\dots, y_m\in \mathfrak S_{k,f,y},$ the field of constants of $k\langle y_1,\dots, y_m\rangle$ is $C.$ \item   For any $y_1,\dots,y_m\in \mathfrak S_{k,f,y},$  $\mathrm{tr.deg}(k\langle y_1,\dots, y_m\rangle|k)=$ $\sum^m_{i=1}\mathrm{tr.deg}(k\langle y_i\rangle|k).$ \end{enumerate}

We prove the following theorem that concerns with the (differential) algebraic dependence of solutions of algebraic ordinary differential equations.

\begin{theorem}\label{algebraicdependence-aode}   Let $k$ be an algebraically closed differential field and $y$ be a nonalgebraic solution of (\ref{aode}).
	\item \begin{enumerate}[(a)] 
		\item \label{nongeneral-SN} $\mathfrak S_{k,f,y}$ is a finite set if and only if 
		there is $k-$irreducible differential subfield $K$ of $k\langle y\rangle$ that can be embedded in a strongly normal extension of $k$ and that  either  a) $K=k\langle w\rangle,$ where $w$ is a nonalgebraic solution of some Riccati differential equation over $k$ or b) $K$ is an abelian extension of $k,$ in which case, $\mathfrak S_{k,f,y}=\{y\}.$\\
		\item\label{auto-nongeneral-SN}  Suppose that $k=C.$ Then, $\mathfrak S_{C,f,y}=\{y\}$  if and only if 
		there is an intermediate differential field $C\subset M\subseteq C\langle y\rangle$ such that either a) $M=C(x),$ where $x'=1$ or $x'=cx$ for some nonzero $c\in C$ or b) $M$ is an abelian extension of $C.$ \\
		\item\label{cardn+2-order4}  Suppose that  $n\leq 4.$  Then, $\mathfrak S_{k,f,y}$ is a finite set if and only if it has at most $n+2$ elements. 
	\end{enumerate}
    \end{theorem}

The above theorem can also be stated and proved for a system of first order differential equations.

From Theorem \ref{algebraicdependence-aode} (\ref{nongeneral-SN}), we obtain that if $y$ is a nonalgebraic solution of a first order differential equation $f(t,t')=0$ then $\mathfrak S_{k,f,y}$ is a finite set if and only if $y$ belongs to an algebraic extension of a strongly normal extension of $k$ (cf. \cite[Theorem 1.1]{KRS24}). It is known that if $y$ is a nonalgebraic solution of the differential equation \begin{equation}\label{extn-Rosexample}g(t,t'):= t'+a_nt^n+\dots+a_3t^3+a_2t^2=0,\end{equation} where $n\geq 3,$ each $a_i\in k$ and there are no elements $z\in k$ such that $z'=a_2$ or $z'=a_3$ then $y$ does not belong to any iterated strongly normal extension of $k$ (\cite[Example 7.4]{KRS24}). Therefore, $\mathfrak S_{k,g,y}$ must be infinite. Observe that the differential equation $h(t,t'):=t'-t^3+t^2$  is a particular case of Equation (\ref{extn-Rosexample}) and thus $\mathfrak S_{k,h,y}$ is an infinite set for any nonconstant solution $y$  of $t'=t^3-t^2,$ as noted in the works \cite{Ros74, Kol-74, NMvdPT22}.

A first order differential equation \begin{equation}\label{fode}g(t,t')=0.\end{equation} is said to be \emph{general type} (see \cite[Page 580]{KRS24}) if it has a nonalgebraic solution $y$ such that $k\langle y\rangle$ contains no nontrivial differential subfield $K$ that can be embedded in a strongly normal extension of $k$. Therefore, from Theorem \ref{algebraicdependence-aode}(\ref{nongeneral-SN}), we obtain that general type first order algebraic differential equations are precisely those equations $f(t,t')=0$ having a nonalgebraic solution $y$ such that  $\mathfrak S_{k,f,y}$ is an infinite set. This confirms the expectation of Noordman et al. (\cite[Page 1667]{NMvdPT22}) on algebraic dependency of solutions of a given general type autonomous first order differential equation. 

As an immediate consequence of Theorem \ref{algebraicdependence-aode} (\ref{auto-nongeneral-SN}) \& (\ref{cardn+2-order4}), we 
obtain that a first order differential equation over $k$ (respectively, over  $C$) is not of general type if and only if it has at most three (respectively, at most one) algebraically independent solutions in any given differential field extension of $k$ (respectively, $C$) having $C$ as its field of constants. This indeed  proves the conjecture made by the authors on algebraic dependency of solutions of nongeneral type equations (\cite[Conjecture 1.2]{KRS24}).

Our proof of the Theorem \ref{algebraicdependence-aode}(\ref{cardn+2-order4}) uses results from \cite{CX18, Pop73, Pop75} on classification of linear algebraic groups having affine homogeneous spaces of dimension $\leq 4.$ For this reason, we need to impose the condition that $n\leq 4.$ At this point of time, we do not know if Theorem \ref{algebraicdependence-aode}(\ref{cardn+2-order4}) generalizes to differential equations of all order.

A solution $y$ of (\ref{aode}) is said to be \emph{generic} if  $\mathrm{tr.deg}(k\langle y\rangle | k)=n,$ the order of the equation. We prove that every nonconstant solution of the Poizat differential equation $f(t,t',t''):=t''-t'h(t),$ where $h$ is a nonzero rational function over $C,$ is generic if and only if $h(t)\neq \partial g(t)/\partial t$ for any rational function $g$ over $C.$ We then use a calculation from \cite[Proposition 4.1]{eagles-Jimenez} and conclude that $\mathfrak S_{C,h,y}$ is an infinite set for any nonconstant solution $y$ of the Poizat equation $t''=t'h(t)$ with $h(t)\neq \partial g(t)/\partial t$ for any $g(t)\in C(t).$ We also provide a necessary and sufficient condition for the Lotka-Volterra (L-V) system to have a generic solution $y$ such that $\mathfrak S_{C,L-V,y}$ is a finite set.  

\subsection{Applications to $d-$solvability of linear differential equations}   Building on the goal initiated by Singer in \cite{Sin85} and \cite{Sin88}, which is to solve a linear homogeneous differential equation $\mathcal{L}(t)=0$ in terms of linear homogeneous differential equations of order $<\mathrm{ord}(\mathcal{L})$, K.A. Nguyen in \cite{Ngu09} develops the notion of $d-$solvable Picard-Vessiot extensions and studies algebraic relations between solutions of a linear differential equation. 

We generalize this notion of $d-$solvabilty to arbitrary differential fields as follows. A differential field  extension $E\supset k$ is said to be \emph{$d^+-$solvable}  if there exists a tower of differential fields $$k=K_0 \subseteq\cdots\subseteq K_r$$ such that $E \subseteq K_r$ and for each $i, 1\leq i\leq r,$ either \begin{enumerate}[(i)]
    \item \label{dsolvable-1} $K_{i}$ is a finite algebraic extension of $K_{i-1},$ or
    \item \label{dsolvable-2}$K_{i}$ is a Picard-Vessiot extension of $K_{i-1}$ for $\mathcal L(y)=b,$   where $\mathcal L\in K_{i-1}[\partial]$ is of order $\leq d$ and $b\in K_{i-1}$.
\end{enumerate} 

A $d^+-$solvable field extension $E$ of $k$ is said to be \emph{$d-$solvable} if (\ref{dsolvable-2}) holds with $b=0.$ A solution $y$ of an algebraic differential equation is said to be \emph{$d-$solvable} (respectively \emph{$d^+-$solvable}) if $k\langle y\rangle$ is a $d-$solvable ( respectively $d^+-$solvable) extension of $k.$  Differential field extensions of $k$ that are $1^+-$solvable are precisely differential subfields of liouvillian extension fields of $k.$ The notion $2-$solvability is studied by Singer in (\cite[page 666]{Sin85}), where he uses the term eulerian instead of $2-$solvable. 

 If a nonzero solution $y$ of a linear homogeneous differential equation $\mathcal{L}(t)=0$ over $k$ is also a solution of a first order algebraic differential equation over $k$, then Goldman and Singer independently prove that such a  $y$ must be $1^+-$solvable; furthermore, if $\mathcal{L} $ is also irreducible, then they show that $\mathcal{L}(t)=0$ has a $C-$basis of $1-$solvable solutions (see \cite[Theorem 3, 4]{GL59} \& \cite[Corollary 3, 4]{Sin76}). Using results on algebraic homogeneous spaces of dimension $\leq 4,$ we reprove and generalize this result as follows.
 
\begin{theorem}\label{sol-td-higher}
   Let $\mathcal L \in k[\partial],$ $y$ be a nonzero solution of  $\mathcal L(t)=0$ and $d:=\mathrm{tr.deg}(k\langle y\rangle| k).$   Then, $y$ is $d^+-$solvable for $d=1$ and $d-$ solvable for $2\leq d\leq 4$. Furthermore, if $1\leq d\leq 4$ and $\mathcal L$ is  irreducible then $\mathcal{L}(t)=0$ has a $C-$basis of $d-$solvable solutions.
\end{theorem}

We conclude this section with the following example that shows that the above theorem need not hold for $d=6.$ Suppose that $E$ is a Picard-Vessiot extension of $k$ with $\mathscr G(E|k)$ isomorphic to the exceptional simple algebraic group $\mathrm{G}_2$. By \cite[Table 2]{BHW86}, $\mathscr G(E|k)$ contains a closed subgroup $\mathscr H$ isomorphic to $\mathrm{SL_3}(C)$ that is  maximal  as well as  reductive  of dimension $8.$ Since $\mathscr H$ is reductive, from \cite[Theorem 1.2.2]{And14}, $E^\mathscr H$ is a solution field and since $\mathscr H$ is maximal, we may choose any $z\in E^\mathscr H\setminus k$ that is  a solution of a linear differential equation over $k$ and obtain that $E^\mathscr H=k\langle z\rangle.$ Note that $\mathrm{dim}(\mathscr H)=8$ and therefore $\mathrm{tr. deg}(k\langle z\rangle|k)=6.$  We will show that $z$ is not $6-$solvable. Assume otherwise and consider a tower of differential fields $k=E_0\subset E_1\subset \cdots\subset E_r,$ where $k\langle z\rangle\subseteq E_r$ and for each $i$ $E_i$ is either a finite algebraic extension of $E_{i-1}$ or a Picard-Vessiot extension with $\mathscr G(E_i|E_{i-1})\subseteq \mathrm{GL}_6.$ Then, we also have the tower of differential fields $$k=E_0\subset E_1\subset \cdots\subset E_r\subset EE_r,$$ where $EE_r$ is a Picard-Vessiot extension of $E_r$ having a Galois group isomorphic to a closed subgroup of $\mathscr H$. Since $\mathscr G(E|k\langle z\rangle)= \mathscr H\cong \mathrm{SL}_3(C)$ is $3-$solvable extension of $k\langle z\rangle$, we obtain that $E$ is a $6-$solvable extension of $k.$ This is a contradiction since it is known that $E$ is not a $6-$solvable extension of $k$ (see \cite[Theorem 1, 2 \& Table 1]{Ngu09}).


\section{Preliminaries}
Unless otherwise explicitly mentioned, $k$ shall denote a differential field of characteristic zero with an algebraically closed field of constants $C$ and all differential fields are finitely generated field extensions of $k$. In this section, we shall introduce a few notations and record several facts from the theory of strongly normal extensions and Picard-Vessiot extensions, from \cite{Kol48, Kol53, Kol55, vdPS03}.  

For a differential field extension  $E$ of $k,$  $C_E$ will denote the field of constants of $E.$  A differential field extension $E$ of $k$ is said to be a \emph{strongly normal extension} of $k$ if the following holds:\begin{enumerate}[(i)]
    \item $E$ is a finitely generated field extension of $k$ with $C_E=C,$\\
    \item (Normality condition) If $\sigma: E\rightarrow M$ is a differential $k-$embedding of $E$ into a differential field extension $M$ of  $E$ then $\sigma E\subseteq  EC_M,$ where $EC_M$ is the field compositum of $E$ and $C_M.$ 	
\end{enumerate}

It is worthy to observe that $k\subseteq E\subseteq \mathcal E$ are differential fields, $E$ is a strongly normal extension of $k$ and $C_\mathcal E=C$ then  every differential $k-$embedding $\sigma: E\to \mathcal E$ becomes a differential $k-$automorphism of $E.$ The group of all $k-$automorphisms of a field extension $K$ will be denoted by $\mathrm{Aut}(K|k)$ and the \emph{differential Galois group} of a differential field extension $K$ of $k,$ denoted by $\mathscr G(E|k),$  is defined as the subgroup of $\mathrm{Aut}(K|k)$ whose elements  commute with the derivation map on $K:$ $$\mathscr G(K|k)=\{\sigma \in \mathrm{Aut}(K|k)\ | \ \sigma(x')=\sigma(x)'\ \text{for all }x \in K\}.$$

The  ring of linear differential operators over $k$ will be denoted by $k[\partial].$  Let $E$ be a differential field extension of $k.$ For a nonempty subset $S\subset E,$ $k\langle S\rangle$ will denote the smallest differential field containing $k$ and $S.$   The set of all elements $x\in E$ such that $\mathcal L(x)=0$ for some $\mathcal L\in k[\partial]$ will be denoted by $T(E|k).$ 

A differential field extension $E$ of $k$ is said to be a \emph{Picard-Vessiot extension} of $k$ for  a monic linear differential operator $\mathcal L\in k[\partial]$ of order $n$ if $C_E=C, E$ contains a $C-$vector space $V$ of dimension $n$ such that $\mathcal L(v)=0$ for all $v\in V$ and that $E=k\langle V\rangle.$ Let $E$ be a Picard-Vessiot extension of $k$ for a monic $\mathcal L\in k[\partial],$ $V\subset E$ be set of all solutions of $\mathcal{L}$ in $E$ and $y_1,\dots, y_n$ be a $C-$basis of $V.$ Let $M$ be a differential field extension of $E,$  $\sigma: E\rightarrow M$ be a differential $k-$embedding and $V_M=\mathrm{span}_{C_M}\{y_1,\dots, y_m\}$. Observe that $\sigma E=k\langle \sigma V\rangle$ and that $\sigma(y_1),\dots, \sigma(y_n)$ is a $C_M-$basis of $V_M$    ( \cite[Corollary 1.13]{vdPS03}). Thus, $V_M$ is the set of all solutions of $\mathcal L(y)=0$ in $M.$ Therefore, $\sigma V\subset V_M.$ This implies $\sigma E \subseteq E C_M$ and we obtain that $E$ is a strongly normal extension of $k.$ 
  
  The action of $\mathscr{G}(E|k)$ on $V$ induces a
  faithful representation  $$\rho: \mathscr G(E|k)\hookrightarrow\mathrm{GL}(V)$$ of groups. It can be shown that the image of the representation $\mathrm{im}(\rho)$ is a Zariski closed subgroup of $\mathrm{GL}(V).$ Thus one can identify $\mathscr G(E|k)$ with the $C-$points of a linear algebraic $C-$group. 
  
  Let $$V_x:=\mathrm{span}_C\{\sigma(x)\ | \ \sigma\in \mathscr G\}.$$ Then it can be shown that 
\begin{enumerate}[(i)]
	\item $T(E|k)=\{x\in E\ |  \ V_x \ \text{is finite dimensional over } \ C\}.$\\
	\item $T(E|k)$ is a differential $k-$algebra whose field of fractions, denoted by $\mathrm{Frac}(T(E|k)),$ is $E$.\\
    \item $T(E|k)$ is generated as a $k-$algebra by $y_1,\dots, y_n, y'_1,\dots, y'_n,\dots,$ $y^{(n-1)}_1,\dots, y^{(n-1)}_n.$
	\end{enumerate}

\subsection*{Galois correspondence} Let $E$ be a strongly normal extension of $k.$ Then, the group $\mathscr G(E|k)$ can be identified with the $C-$points of an algebraic $C-$group $\mathscr{G},$ which is not necessarily linear,  in such a way that if $E$ is a Picard-Vessiot extension of $k$ then the linear algebraic group $\mathrm{im}(\rho)$ is isomorphic to $\mathscr{G}.$ Furthermore, $\mathscr{G}$ is a linear algebraic $C-$group if and only if $E$ is a Picard-Vessiot extensions of $k$ (\cite[Example 2 \& Theorem 2]{Kol55}). There is a bijective (Galois) correspondence between the intermediate differential subfields of a strongly normal extension $E$ of $k$ and  Zariski closed subgroups  of $\mathscr{G}(E|k)$ given by the maps
\begin{align*}
	&K \to \mathscr{G}(E|K):=\left\lbrace \sigma \in \mathscr{G}(E|k)\ |\ \sigma(u)=u \  \ \text{for all}\  \ u  \in K \right\rbrace \\
	& \mathscr{H} \to E^{\mathscr{H}}:= \left\lbrace u  \ \in E\ |\ \sigma(u)=u \ \ \text{for all} \ \sigma  \in \mathscr{H} \right\rbrace,
\end{align*}
where $\mathscr H$ is a Zariski closed subgroup of $\mathscr{G}(E|k)$ and $K$ is a differential field intermediate to $E$ and $k.$ In this Galois correspondence,  a differential  field $K$ is algebraically closed in $E$  if and only if the closed subgroup $\mathscr G(E|K)$ is  connected;  in particular, if $k$ is an algebraically closed field then $\mathscr G(E|k)$ is a connected. The codimension of $\mathscr G(E|K)$ (as a variety) is equal to  $\mathrm{tr.deg(K|k)}$. An intermediate differential field  $k\subseteq K\subseteq E$ is a strongly normal extension of $k$ if and only if the subgroup $\mathscr G(E|K)$ is normal in $\mathscr G(E|k);$ in which case $\mathscr G(E^{\mathscr G(E|K)}|k)\cong\mathscr G(E|k)/\mathscr G(E|K).$ 

Let $E$ be a Picard-Vessiot of $k.$ If $K$ is a differential subfield of $E$ then $T(K|k)=T(E|k)\cap K$ and therefore $T(K|k)$ is a differential $k-$algebra. Though $E$ has the property that $\mathrm{Frac}(T(E|k))=E,$ its intermediate differential subfields may not have this property.  In fact, it may happen that $\mathrm{Frac}(T(K|k))=k.$  In \cite{And14}, Yves Andr\'{e} shows that a differential fields $K$ intermediate to $E$ and $k$ has the property $\mathrm{Frac}(T(K|k))=K$  if and only if $\mathscr G(E|K)$ is observable in $\mathscr G(E|k).$

 We shall conclude this section with the following proposition and a remark.

\begin{proposition} \label{solnoveralb.extn}
	Let $k\subseteq \tilde{k} \subseteq K\subset E$ be differential fields, $\tilde{k}$ be an algebraic extension of $k$ and $E$ be a Picard-Vessiot extension of $k.$ Then  $T(K|k)=T(K|\tilde{k}).$
	\end{proposition} 
\begin{proof} We only need to show $T(K|\tilde{k})\subseteq T(K|k).$
	Let $k^0$ be the algebraic closure of $k$ in $E.$ Then from \cite[Proposition 2.1]{URVRS-21}, $T(k^0K|k^0)=T(k^0K|k).$ Since $\tilde{k}\subseteq k^0,$ we have $T(k^0K|k)=T(k^0K|\tilde{k}).$ Therefore, 
	$$T(K|k)=T(k^0K|k)\cap K=T(k^0K|\tilde{k})\cap K=T(K|\tilde{k}).$$ 
	\end{proof}

\begin{remark} \label{embeddingSNafterfinitalg} Let $k\subseteq \tilde{k} \subseteq K$ be differential fields and $C_K=C$ and $\tilde{k}$ be an algebraic extension of $k.$ 

\begin{enumerate}[(i)] \label{PVafterfinite} \item Let $F$ be a finite Galois extension of $k$ containing $\tilde{k}$ and contained in an algebraic closure of $K.$ Let  $y\in T(K|\tilde{k})\setminus \tilde{k}$ such that $\mathcal{L}(y)=0$ for some $\mathcal L\in \tilde{k}[\partial].$ Let $E$ be a Picard-Vessiot extension of $F$ for $\mathcal L.$ We then have $k\subseteq F\subseteq E,$ where $F$ is a Picard-Vessiot (finite Galois) extension of $k$ and $E$ is a Picard-Vessiot extension of $F.$ Applying \cite[Proposition 5]{Juan-Magid-08}, we obtain that $E$ can be embedded in a Picard-Vessiot extension $E_1$ of $k.$ Thus we have $k\subseteq \tilde{k} \subseteq \tilde{k}\langle y\rangle \subseteq E_1.$ Now, applying Proposition \ref{solnoveralb.extn}, we obtain that $y\in T(K|k).$ Form this it follows that $T(K|\tilde{k})=T(K|k).$\\
\item\label{SNafterfinitealg} If $K$ is a strongly normal extension of $\tilde{k}$ then from \cite[Proposition 20]{Kov-73}, it can be seen that $K$ is contained in a union of strongly normal extensions of $k.$ Since $K$ is finitely generated field extension of $k,$ we obtain that $K$ is contained in some strongly normal extension of $k.$
\end{enumerate}
	\end{remark}


\section{Differential subfields of Picard-Vessiot extensions}

This section is devoted to proving Theorem \ref{intromain}.  The proof is essentially broken up in to several lemmas and propositions. We shall first prove a field-theoretic proposition which is crucial to the proof of Theorem \ref{intromain}.  

\begin{proposition}\label{Field-descent}
Let $k\subset E$ be field extensions and let $K, L, M$ be intermediate subfields such that  $K$ and $M$ are free over $k,$ $M\subset L\subseteq KM$ and that $\mathrm{tr.deg}(L|M)\geq 1.$ Suppose that $K$ is the fixed field of a subgroup $\mathscr{G}$ of $\mathrm{Aut}(KM|K)$ and that $L$ is a stable under the group $\mathscr G.$  Then,  there is an element $y\in (K\setminus k)\cap L$. If $L=M(\eta),$ $\eta$ is transcendental over $M$ and  $\sigma(\eta)=a_\sigma \eta$ for all $\sigma\in \mathscr{G}$ and $a_\sigma\in M$ then there is an element $x\in M$ and a positive integer $l$ such that $x\eta^l\in (K\setminus k)\cap L.$ 
\end{proposition}

\begin{proof}
\sloppy Let $x_1,x_2,\dots,x_n$ be a transcendence base of $K$ over $k$.   Then since $K$ and $M$ are free over $k$ and that $\mathrm{tr.deg}(L|M)\geq 1,$ the elements $x_1,x_2,\dots,x_n$ must be algebraically dependent over $L.$   Let $P\in L[X_1,\dots,X_n]$ be a polynomial with minimal number of terms such that $P(x_1,\dots,x_n)=0$. We may assume that 
$$P=\sum^t_{i=1}m(\alpha_{1i},\dots,\alpha_{ni})X^{\alpha_{1i}}_1\cdots X^{\alpha_{ni}}_n,$$
where for each $i,$ $m(\alpha_{1i},\dots,\alpha_{ni})\in L$ and $m(\alpha_{11},\dots,\alpha_{n1})=1.$  From the algebraic independence of $x_1,\dots, x_n$ over $k$, it is clear that there is a integer $j$ such that $m(\alpha_{1j},\dots,\alpha_{nj})\in L\setminus k$.  For $\sigma\in \mathscr{G},$ let $$P_{\sigma}=\sum^t_{i=1}\sigma(m(\alpha_{1i},\dots,\alpha_{ni}))X^{\alpha_{1i}}_1\cdots X^{\alpha_{ni}}_n$$  and 
let $Q_{\sigma}=P_{\sigma}-P$. Then, since $\mathscr{G}$ stabilizes $L,$ both $P_\sigma$ and $Q_\sigma$ belong to $L[X_1,\dots, X_n].$ Now, $P(x_1,\dots,x_n)=0$ implies $P_{\sigma}(x_1,\dots,x_n)=0$ and therefore $Q_{\sigma}(x_1,\dots,x_n)=0$. Since $m(\alpha_{11},\dots,\alpha_{n1})=1,$ the polynomial   $Q_{\sigma}\in L[X_1,\dots,x_n]$ has fewer nonzero terms than $P$ and therefore $Q_\sigma=0$. That is, $\sigma(m(\alpha_{1i},\dots,\alpha_{ni}))=m(\alpha_{1i},\dots,\alpha_{ni})$ for all $i=1,2,\dots, t.$ Since $K$ is the fixed field of $\mathscr{G},$ we see that all the coefficients of $P_\sigma$ belong to $K.$   In particular, $m(\alpha_{1j},\cdots,  \alpha_{nj})\in (K\setminus k)\cap L.$ 

Suppose that $L=M(\eta)$ and $\sigma(\eta)=a_\sigma \eta$ for all $\sigma\in \mathscr{G}$ and $a_\sigma\in M.$ Choose   $y\in (K\setminus k)\cap L$ and write $$y=
\frac{P}{Q},$$ where $P,Q\in M[\eta],$ $(P,Q)=1$ and $Q$ is monic.  Then for any $\sigma\in \mathscr G,\sigma(y)=y$ implies $$\sigma(P)Q=\sigma(Q)P.$$ Since, $\sigma(\eta)=a_\sigma \eta,$ it follows that $\sigma(P)=r_{\sigma}P$ and $\sigma(Q)=r_{\sigma}Q$ for some $r_\sigma\in M.$ If both $P$ and $Q$ have only one term then it easily follows that $y=x\eta^m$ for some $x\in M$ and a nonzero integer $m.$ Thus, either $y$ or $1/y$ has the desired form. Therefore, we may assume $P$ has more than two terms and write $$P=a_n\eta^n+a_{n-1}\eta^{n-1}+\cdots+a_0,$$ where
$a_0,\dots, a_n\in M$ with $a_n\neq 0$ and $a_i\neq 0$ for some $i.$ Then, since $\sigma(\eta)=a_\sigma \eta$ and that $\sigma(P)=r_{\sigma}P,$ it follows that $\sigma (a_n\eta^n)=r_\sigma a_n\eta^n$ and $\sigma (a_i\eta^i)=r_{\sigma} a_i\eta^i.$ This implies $x\eta^l\in (K\setminus k)\cap L,$ where $x=a_n/a_i$ and $l=n-i.$\end{proof}

\begin{lemma}\label{unipotent-case} 
	Let $E$ be a Picard-Vessiot extension of $k$ and  $\mathscr{U}$ be the unipotent radical of $\mathscr{G}.$ Suppose that $K$ is a differential subfield intermediate to $k$ and $E$ such that  $E^{\mathscr{U}}\cap K=k.$ Then, there are an element $\zeta \in T(K|k)\setminus k,$ and monic operators  $\mathcal L_1, \mathcal D\in k[\partial],$ where  $\mathrm{ord}(\mathcal L_1)=1$ and $\mathcal D$ is irreducible, such that $\mathcal L_1(\mathcal D(\zeta))=0,$ $\mathcal D(\zeta)\neq 0$ and that  $\mathrm{ord}(\mathcal L_1\mathcal D)\leq \mathrm{tr.deg}(K|k)+1.$ 
	\end{lemma}
	\begin{proof} The proof of the lemma consists of four parts. In the first part, we will show that 	$T(K|k)\setminus k\neq \emptyset.$ In the second part, we will show the existence of monic operators $\mathcal L_1, \mathcal D\in k[\partial],$ where $\mathrm{ord}(L_1)=1$ and $\mathcal D$ is irreducible, and an element $\zeta\in T(K|k)\setminus k$ such that $\mathcal D(\zeta)\neq 0$ and that $\mathcal L_1(\mathcal D(\zeta))=0.$  In the third part, we find a positive integer $m\leq \mathrm{tr.deg}(K|k)$ such that $\zeta=u_0+\sum^m_{i=1}u_i\eta_i,$ where $u_0,u_1,\dots,u_m\in E^\mathscr U$  and that $\eta'_i\in E^\mathscr U$ for $i=1,\dots,m.$ The last part of the proof  uses this description of $\zeta$ to show that  $\mathrm{ord}(\mathcal D)\leq m.$    
	
	We first observe that $\mathscr U\subseteq \mathscr G^0,$ the condition $E^{\mathscr{U}}\cap K=k$ implies that $k$ is algebraically closed in $K.$ 
	Next, by Galois correspondence, $E^{\mathscr{U}}$ is a Picard-Vessiot extension of $k$ with differential Galois group $\mathscr{G}/\mathscr U.$ Since $K\cap E^\mathscr U=k,$  $KE^{\mathscr{U}}$ is also a Picard-Vessiot extension of $K$ whose differential Galois group is isomorphic to $\mathscr {G}/\mathscr{U}.$ Furthermore, the isomorphism map is nothing but the restriction map:  $$\mathscr G(KE^{\mathscr{U}}|K)\cong \mathscr G(E^\mathscr{U}|E^\mathscr{U}\cap K)=\mathscr G(E^\mathscr{U}|k).$$ 
	
	Thus, $\mathrm{ tr. deg}(K|k)=\mathrm{ tr. deg}(KE^{\mathscr{U}}|E^\mathscr{U})$ and therefore $K$ and $E^\mathscr{U}$ are free over $k.$ Since $k$ is algebraically closed in $K,$  $E^{\mathscr{U}}$  and $K$ are linearly disjoint over $k$. Note that $E$ is a Picard-Vessiot extension of $E^{\mathscr U}$ with $\mathscr U$ as its differential Galois group, it follows from \cite[Theorem 4.4]{KRS24} that there is an element $\eta\in KE^\mathscr U\setminus E^{\mathscr U}$ such that $\eta'\in E^{\mathscr U}.$

    \begin{minipage}{.7\textwidth}  Let $\mathcal A=\{\eta\in KE^{\mathscr U}\ | \ \eta'\in E^\mathscr U\}$ and $L$ be the field generated over $E^\mathscr U$ by $\mathcal A.$ Then $L$ is a differential field and it can be shown that $L=E^\mathscr U(\eta_1,\dots, \eta_m),$ where $\eta_1,\dots, \eta_m\in \mathcal A$ are algebraically independent over $E^\mathscr U$.  As noted earlier, every element $\sigma\in \mathscr G(KE^{\mathscr{U}}|K)$ is a  differential $k-$automorphism of $E^\mathscr U.$ Therefore for $\eta\in \mathcal A$ we have $\sigma(\eta)'=\sigma(\eta')\in E^{\mathscr U}$  and thus $\mathcal A$ is a $\mathscr G(KE^{\mathscr{U}}|K)-$stable set. Consequently,  we obtain that $L$ is  a $\mathscr G(KE^{\mathscr{U}}|K)-$stable (differential) field. Applying Proposition \ref{Field-descent}, we obtain an element $y\in (K\setminus k)\cap L$.  We also note there that $L$ is a Picard-Vessiot extension of $E^{\mathscr U},$ whose  differential Galois group isomorphic to the commutative vector group $\prod^m\mathrm G_a.$ 
\end{minipage} \begin{minipage}{.3\textwidth}
	 	
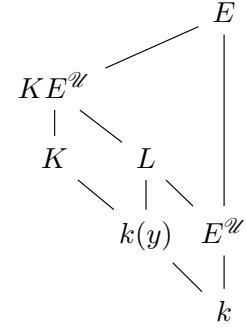
\hspace{1em}\begin{tikzcd}[row sep=2ex, column sep= .05ex]
	 		& &E\ar[dash]{ddd}\ar[dash]{dll}\\
	 	KE^{\mathscr U}\ar[dash]{dr}\ar[dash]{d}& &\\
	 		K \ar[dash]{dr}& L \ar[dash]{dr}\ar[dash]{d}& \\
	 		& k(y)\ar[dash]{dr}& E^{\mathscr U}\ar[dash]{d}\\
	 		&& k
	 	\end{tikzcd}
	 \end{minipage}
 
 The action of the automorphisms from $\mathscr G(KE^\mathscr U|K)$ on $L$ can be explicitly calculated. For any $\sigma\in \mathscr G(KE^\mathscr U|K)$ and for any $i=1,2,\dots, m,$ we see that $\sigma(\eta_i), \eta_1,\dots, \eta_m$ are algebraically dependent over $E^\mathscr U.$ By Kolchin-Ostrowski Theorem, we obtain that $$\sigma(\eta_i)=f_{\sigma}+\sum^m_{j=1}c_{ij\sigma}\eta_j,$$ where $c_{ij\sigma}$ are constants and $f_\sigma\in E^{\mathscr U}.$ Thus, the differential ring $E^{\mathscr U}[\eta_1,\dots,\eta_m]$ is also $\mathscr G(KE^\mathscr U|K)-$stable. If we write $y=P/Q,$ for $P, Q\in E^{\mathscr U}[\eta_1,\dots,\eta_m]$ with $(P,Q)=1$  then $\sigma(y)=y$ implies $\sigma(P)Q=\sigma(Q)P$ and we obtain $\sigma(P)=r_\sigma P$ and $\sigma(Q)=r_\sigma Q$ for some $r_\sigma\in E^{\mathscr U}.$  
 
 We may assume $P\notin E^{\mathscr U}$ and observe that the field $E^{\mathscr U}(P)$ is $\mathscr G(KE^\mathscr U|K)-$stable. Thus, again applying Proposition \ref{Field-descent}, we obtain an element $\zeta:=xP^l\in K\setminus k$ for some $x\in E^{\mathscr U}.$ Note that $\zeta\in E^{\mathscr U}[\eta_1,\dots,\eta_m]=T(L|E^{\mathscr U})$ and in particular, for some positive integer $n,$ $\zeta, \zeta',\dots, \zeta^{(n)}$ are linearly dependent over $E^{\mathscr U}.$ Now since $K$ and $E^\mathscr U$ are linearly disjoint over $k,$ we obtain that $\zeta, \zeta',\dots, \zeta^{(n)}$ are linearly dependent over $k.$ This shows that $k\subset T(K|k)$ and thus the first part of the proof is complete.
 
 Choose a monic operator $\mathcal L\in k[\partial]$ of smallest order such that $\mathcal L(\zeta)=0$ for some  $\zeta\in (T(K|k)\setminus k)\cap E^{\mathscr U}[\eta_1,\dots,\eta_m].$ We shall now show that $$\zeta=u_0+\sum^m_{i=1}u_i\eta_i$$ where $u_0,u_1,\dots,u_m\in E^\mathscr{U}.$
  Let $V\subset E$ be the set of all solutions of $\mathcal L(t)=0.$ Note that $\mathcal L(t)=0$ has  a nonzero solution in $E^\mathscr U$ (see \cite[Lemma 6]{Sin-Dav} or \cite[Proposition 2.2]{Sri20}). Let $W=V\cap E^{\mathscr U}.$ Since  $E^{\mathscr U}$ is a Picard-Vessiot extension of $k,$ $W$ is $\mathscr G-$stable. Let $W_1$ be any nonzero $\mathscr G-$stable subspace of $W.$ Then, $$\mathcal L= \mathcal L_1 \mathcal L_{W_1},$$ $\mathcal L_1, \mathcal L_{W_1}\in k[\partial]$ and $W_1$ is the full set of solutions of $\mathcal L_{W_1}.$ Note that $\mathrm{ord}(\mathcal L_{W_1})=\mathrm{dim}_C(W_1).$ Since $\zeta\notin W,$  $\mathrm{ord}(\mathcal L_1)<\mathrm{ord} (\mathcal L)$ and $\mathrm{ord}(\mathcal L_{W_1})<\mathrm{ord}(\mathcal L).$ As $\mathcal L(\zeta)=0$ and $\mathcal L_{W_1}(\zeta)\in K,$ from the minimality of $\mathcal L,$ we obtain that $\mathcal L_{W_1}(\zeta)=b\in k\setminus \{0\}.$ By homogenizing this inhomogeneous equation, we obtain that $\mathrm{ord}(\mathcal L)= \mathrm{ord}(\mathcal L_{W_1})+1.$  Observe that $$V=C\zeta\oplus W_1,$$  
  $W\subset V,$ $W_1\subseteq W$ and that $\zeta\notin W.$ Therefore, $W_1=W$ and in particular, $W$ is an irreducible $\mathscr G-$module.  Thus, we have shown the existence of monic operators $\mathcal L_1, \mathcal L_W\in k[\partial]$ such that \begin{equation}\label{orderofL}\mathcal L= \mathcal L_1 \mathcal L_{W},\end{equation}  
  where $\mathcal L_W(\zeta)=b\in k\setminus \{0\},$ $\mathcal L_W$ is irreducible, $\mathrm{ord}(\mathcal L_1)=1$  and that $\mathcal L(\zeta)=0.$ This completes the proof of the second part. 
  
  For an  $\eta\in \{\eta_1,\dots, \eta_m\},$ let $S_\eta=\{\eta_1,\dots, \eta_m\}\setminus \{\eta\}$ and  $\zeta=a_l\eta^l+a_{l-1}\eta^{l-1}+\cdots+a_0,$ where $a_i\in E^{\mathscr U}(S_\eta)$ and $a_l\neq 0.$ 
  We will show that $a_l\in W \subset E^{\mathscr U}$ and that $l=1.$ This will then prove that $\zeta=u_0+\sum^m_{i=1}u_i\eta_i,$ where each $u_i\in E^{\mathscr U}.$
  
  From \cite[Proposition 2.2]{Sri20},  we obtain that $\mathcal L(a_l)=0$ and thus  $a_l=c\zeta+w$ for  $c\in C$ and $w\in W.$ Now $a_l\in V\cap E^{\mathscr U}(S_\eta)$ implies $c=0$ and we obtain $a_l\in W\subset E^\mathscr U.$ For $c\in C\setminus \{0\},$ consider the differential $E^{\mathscr U}(S_\eta)-$automorphism $\sigma_c $ of $L=E^\mathscr U(\eta_1,\dots,\eta_m)$ defined by, $$\sigma_c(\eta)=\eta+c, \quad \sigma_c(\eta_i)=\eta_i \ \text{for all }\ \eta_i\neq \eta\quad\text{and }\ \sigma_c(u)=u\ \text{for all }\ u\in E^{\mathscr U}.$$ Then, since $\sigma_c$ is also $k-$linear, $\sigma_c(\zeta)\in V$ and therefore, $\sigma_c(\zeta)-\zeta\in V$ as well.  Moreover, it easily seen that \begin{equation}\label{degetaisone}\sigma_c(\zeta)-\zeta=la_lc\eta^{l-1}+ \text{lower powers of } \eta.\end{equation}  
  
 For every $v\in V,$ we have $v=c(a_l\eta^l+a_{l-1}\eta^{l-1}+\cdots+a_0)+w,$ where $c\in C$ and $w\in E^{\mathscr U}.$ Thus have the following 
 \begin{enumerate}[(i)]\item $v$ has degree $l$ in $\eta$ if and only if $c\neq 0$ 
	\item $v\in W\subset E^{\mathscr U}$ if and only if $c=0.$ \end{enumerate}  Now, we obtain from Equation (\ref{degetaisone}), that $l=1$ and that
  $\zeta=u_0+\sum^m_{i=1}u_i\eta_i,$ where $u_i\in E^{\mathscr U}$ for $i=0,\dots,m.$ This completes the thrid part of the proof.
  
  Finally, let $\zeta=u_0+\sum^m_{i=1}u_i\eta_i,$ $u_0, \dots, u_m\in E^{\mathscr U}.$ Again by \cite[Proposition 2.2]{Sri20}, we see that $u_1,\dots,u_n\in W.$  For any $\sigma\in \mathscr G(KE^{\mathscr U}|K),$ we know that $$\sigma(\eta_i)=f_{i\sigma}+\sum^m_{j=1}c_{ij\sigma}\eta_j,$$ where $c_{ij\sigma}$ are constants and $f_{i\sigma}\in E^{\mathscr U}.$ Since $\eta_1,\dots, \eta_m$ are algebraically independent over $E^{\mathscr U},$ it can be seen that the product of the matrices $(c_{ij\sigma]})$ and $(c_{ij\sigma^{-1}})$ is the identity matrix. That is $(c_{ij\sigma]})^{-1}=(c_{ij\sigma^{-1}}).$ Since $\zeta\in K,$ we have $\sigma(\zeta)=\zeta$ and we obtain $$\sigma(u_i)=\sum^m_{j=1}c_{ji\sigma^{-1}}u_j$$ for all $\sigma\in \mathscr G(KE^{\mathscr U}|K).$ Now since $\mathscr G(E^{\mathscr U}| k)\cong \mathscr G(KE^{\mathscr U}|K)$ and the isomorphism is given by the restriction map, we obtain that $U=\mathrm{span}_C\{u_1,\dots,u_m\}$ is $\mathscr G(E^{\mathscr U}| k)-$stable $C-$ subspace of $W.$  The restriction map from $\mathscr G$ to $\mathscr G(E^{\mathscr U}| k)$ is also surjective and therefore $U$ becomes a $\mathscr G-$stable subspace of $W.$ We noted in the second part of the proof that $W$ is an irreducible $\mathscr G-$module. Thus, $U=W$ and we have   $$\mathrm{ord}(\mathcal L_{W})=\mathrm{ord}(\mathcal L_U)=\mathrm{dim}_C(U)\leq m \qquad \text{and}\quad \mathcal L_W(\zeta)\in k\setminus\{0\}.$$ 
  
  Now from Equation (\ref{orderofL}),  we obtain $\mathrm{ord}(\mathcal L)\leq m+1$ and this completes the proof of the final part of the lemma.
  \end{proof}         

\begin{lemma}\label{torus-case}
	Let $E$ be a Picard-Vessiot extension of $k,$ $\mathscr G:=\mathscr G(E|k)$ be a connected  reductive linear algebraic group and $\mathscr{R}$ be its radical.  Let $K$ be a differential subfield intermediate to $k$ and $E.$ Suppose that  $ E^{\mathscr{R}}\cap K=k.$ Then  $k\subset T(K|k)$  and there is an element $\zeta\in T(K|k)\setminus k$ such that $\zeta'/\zeta\in k.$ \end{lemma}     
	
	\begin{proof}
	Since $\mathscr G$ is connected and reductive, we have that $\mathscr R$ is a torus, $E$ is a Picard-Vessiot extension over $E^\mathscr R$ having $\mathscr R$ as its differential Galois group and $E^{\mathscr R}$ is algebraically closed in $E$. By Galois correspondence, it follows that  $KE^{\mathscr R}$ is a Picard-Vessiot extension of $E^\mathscr R$ and that $$KE^{\mathscr R}=E^{\mathscr R}(\eta_1,\dots, \eta_m),$$ where $\eta_1,\dots, \eta_m$ are  algebraically independent elements\footnote{Every intermediate differential field is Picard-Vessiot over $E^\mathscr R$ and are generated by exponentials of $E^{\mathscr R}$.} over $E^\mathscr R.$ Since $K\cap E^\mathscr R=k,$ as noted in Lemma \ref{unipotent-case}, the restriction morphism from $\mathscr G(KE^\mathscr R|K)$ to $\mathscr G(E^{\mathscr R}|k)$ is an isomorphism. The group $\mathscr G(KE^\mathscr R|K)$ is a (connected) semisimple algebraic group and every automorphism $\sigma\in \mathscr G(KE^\mathscr R|K),$ stabilizes the field $E^{\mathscr R}.$  Thus, for each $i=1,2,\dots, m$ and $\sigma \in \mathscr G(KE^\mathscr R|K),$ $\sigma(\eta_i)'/\sigma(\eta_i)\in E^{\mathscr R}.$ It is known that $T(E^\mathscr R(\eta_1,\dots, \eta_m)|k)=E^{\mathscr R}[\eta^{\pm1}_1, \dots, \eta^{\pm 1}_m]$ (\cite{Mag94}) and that for each $i=1,2,\dots, m,$ $\sigma(\eta_i)$ is a product of an element of $E^{\mathscr R}$ and a power product of $\eta_1,\dots, \eta_m$ (\cite[Theorem 1]{Ros-75}): $$\sigma(\eta_i)=f_{i\sigma}\eta^{\gamma_{1i\sigma}}_1\eta^{\gamma_{2i\sigma}}_2\cdots\eta^{\gamma_{mi\sigma}}_m,$$where  each $\gamma_{ji\sigma}$  is an integer and $f_{i\sigma}\in E^{\mathscr R}.$ Consider the map  \begin{equation*}
		\Phi: \mathscr{G}(KE^\mathscr R| K) \rightarrow \mathrm{GL}_m(\mathbb{Z}): \quad \sigma \longmapsto (\gamma_{ij\sigma})_{m\times m}.
	\end{equation*}
Since $\eta_1,\dots, \eta_m$ are algebraically independent,  it follows that for $\sigma, \tau\in \mathscr G(KE^{\mathscr R}| K),$ $$\gamma_{ij\sigma\tau}=\sum^m_{p=1}\gamma_{ip\tau}\gamma_{pj\sigma}.$$ Thus, the map $\Phi$ is a group homomorphism.  Since $\mathscr G(KE^{\mathscr R}| K)$ is a connected  linear algebraic group, $\Phi$ must be the trivial homomorphism (\cite[Theorem]{KM79}). Thus, for each $i=1,2,\dots,m $ and $\sigma\in \mathscr G(KE^{\mathscr U}| K),$ $$\sigma(\eta_i)=f_{i\sigma}\eta_i.$$

Now we shall apply Proposition (\ref{Field-descent}) and obtain an element for each $i=1,2,\dots, m,$ $x_i\eta^{l_i}_i\in K\setminus k,$ where $x_i\in E^{\mathscr R}$ and positive integers $l_i.$  Note that $$(x_i\eta^{l_i}_i)'/x_i\eta^{l_i}_i=x'_i/x_i +l_i(\eta'_i/\eta_i)\in E^{\mathscr R}\cap K=k.$$
\end{proof}

We require the next three propositions to deal with the differential subfields of Picard-Vessiot extensions whose differential Galois groups are semisimple algebraic groups. These results may also be of independent interest.
\begin{proposition} \label{algebraicsolution-irreduciblefield} Let $k\subset K\subseteq E$ be differential fields such that $C_E=C.$  Suppose that $\mathcal L\in k[\partial]$ is a monic operator such that for some $y\in E,$ $\mathcal L(y)=0,$ $y$ is nonalgebraic over $k$ and $y$ algebraic over $K.$  Then,  there is an element $x\in K$ such that $x$ is nonalgebraic over $k$ and that $\mathrm{sym}^d(\mathcal L)(x)=0,$ where  $\mathrm{sym}^d(\mathcal L)\in k[\partial]$ is the $d^{th}$ symmetric power of $\mathcal L$ for some integer $d\geq 1.$ 
\end{proposition}
\begin{proof} Let $L$ be the splitting field of the monic irreducible polynomial of $y$ over $K$ and  $\mathscr{G}(L|K)=\{e=\sigma_1, \sigma_2,\dots, \sigma_n\}$ be the Galois group of $L$ over $K.$ Then the $K$-automorphism are also differential $K-$automorphisms. Thus, if $\mathcal L(y)=0$ for some $\mathcal L\in k[\partial]$ then  for each $i,$ $\mathcal L(\sigma_i(y))=0.$ Let $p(X)=X^n+a_{n-1}X^{n-1}+\cdots+a_0\in K[X]$  be the monic irreducible polynomial of $y$ over $K.$ Then for each $i=1,2,\dots, n,$ there are symmetric polynomials $s_i$ in $n$ variables such that $a_i=s_i(y,\sigma_2(y),\dots, \sigma_n(y)).$ Clearly, for each $i=1,2,\dots, n,$ $\mathrm{sym}^{i}(\mathcal L)(a_{n-i})=0.$ 
Finally, since $y$ is assumed to be nonalgebraic over $k,$ there exists an  $i, 1\leq i\leq n$ such that $x:=a_i$ is not algebraic over $k.$  
\end{proof}
\begin{proposition}\label{algebraicriccatisolution-irreduciblefield} Let $k\subseteq K\subseteq E$ be differential field such that $E$ is a Picard-Vessiot extension of $k,$ $\mathscr G:=\mathscr G(E|k)$ and $\mathscr G^0$ be a semisimple algebraic group.  Suppose that $\mathcal L\in k[\partial]$ be a monic operator such that for some nonzero $y\in E,$ $\mathcal L(y)=0,$ $y$ is not algebraic over $k,$  and that  $y'/y$ is algebraic over $K.$ Then, there is a nonzero element $x\in E$ such that  $x'/x$ is not algebraic over $k,$ $x'/x\in K$ and that $\mathrm{sym}^d(\mathcal L)(x)=0$ for some integer $d\geq 1.$ \end{proposition}
\begin{proof}
Let $P(X)=X^n+a_{n-1}X^{n-1}+\cdots+a_0\in K[X]$  be the monic irreducible polynomial of $y'/y$ over $K.$ Since $E$ is also a Picard-Vessiot extension of $K$, it follows that $E$ contains a splitting field $L$ of $p$. Let $\mathscr G(L|K)=\{e=\sigma_1, \sigma_2,\dots, \sigma_n\}$ be the ordinary Galois group of $L$ over $K.$ Since, the restriction map $\mathscr G(E|K)\to \mathscr G(L|K)$ is surjective, there are differential automorphisms $e, \hat{\sigma}_2, \dots, \hat{\sigma}_n\in \mathscr G(E|K)$ whose restrictions are $e, \sigma_2,\dots, \sigma_n,$ respectively.
Now, $$a_{n-1}=\sum^n_{i=1}\sigma_i(y'/y)=\frac{x'}{x},$$ where $x=\prod^n_{i=1}\hat{\sigma}_i(y).$ Observe that $\mathrm{sym}^n(\mathcal L)(x)=0$ and that $x'/x\in K.$ Let $k^0$ be the algebraic closure of $k$ in $E.$ We claim that $x\notin k^0$ and that $x'/x\notin k^0,$ which then proves the proposition. First, we suppose on the contrary that $x\in k^0.$  Then, we have $y\hat{\sigma}_1(y)\cdots\hat{\sigma}_n(y)=:\alpha\in k^0.$ Since $T(E|k)$ is a $\mathscr G(E|k)-$stable $k-$algebra, we have $\hat{\sigma}_1(y)\cdots\hat{\sigma}_n(y)\in T(E|k).$ Therefore $\alpha/y\in T(E|k).$ Every element of $k^0$ is a solution of some linear homogeneous differential equation over $k$ and therefore $1/y\in T(E|k).$ That is, $y$ and $1/y$ are solutions of some linear homogeneous differential equation over $k.$ From \cite[Theorem]{Har-Sib}, it then follows that $y'/y\in k^0.$ Since, by assumption, $y\notin k^0,$ the differential field $k^0(y)$ is a Picard-Vessiot extension of $k^0$ having a differential Galois group isomorphic to $\mathrm{G}_m.$ 
This in turn implies that there is a surjective morphism of algebraic groups $\mathscr G^0\to \mathrm{G}_m.$ Now since $\mathscr G^0$ is a semisimple algebraic group, we have obtained a contradiction. Thus, $x\notin k^0.$ Now, if $x'/x\in k^0$ then $k^0(x)$ is a Picard-Vessiot extension of $k^0$ having a differential Galois group isomorphic to $\mathrm{G}_m,$ which again contradicts the fact that there are no nontrivial morphisms from $\mathscr G^0$ to $\mathrm{G}_m.$ This proves our claim.  
\end{proof}

\begin{proposition}\label{reductiontosimplealgebraicgroup} Let $E$ be a Picard-Vessiot extension of $k,$ $\mathscr G:=\mathscr G(E|k)$ be a (connected) simple algebraic group, $K$ be a differential field intermediate to $k$ and $E.$ Suppose that $K=E^\mathscr P,$ where $\mathscr P$ is a parabolic subgroup of $\mathscr G.$ Then $k\subset K\subset \tilde{E}\subset E$  such that $\tilde{E}$ is a Picard-Vessiot extension of $k,$ $\mathscr G(\tilde{E}|k)$ is isogenous to $\mathscr G(E|k)$ and there is a finite dimensional $C-$vector space $U\subset \tilde{E}$ such that $U$ is a faithful and an irreducible $\mathscr G(\tilde{E}|k)-$module. \end{proposition}
	
\begin{proof} As $G$ is a differential Galois group of a Picard-Vessiot extension, there exists a faithful $\mathscr G-$module $V\subset E.$ If $V$ is irreducible the proposition is easily proved by taking $\tilde{E}=E.$ Therefore, we shall assume $V$ is reducible. As $\mathscr G$ is semisimple algebraic group, we choose an irreducible $\mathscr G-$submodule $U.$ Note that $U$ may not be a faithful $\mathscr G-$module.  However, if the representation is given by $\phi: G\to \mathrm{GL}(U)$ then $U$ is a faithful and an irreducible $G/\mathrm{ker}(\phi)-$module.  Since $\mathscr G$ is simple,  $\mathrm{ker}\phi$ is  a finite normal subgroup of $\mathscr G.$ Let $\tilde{E}:=E^{\mathrm{ker}(\phi)}$ and observe that $\tilde{E}$ is a Picard-Vessiot extension of $k,$ $U\subset \tilde{E}$ and $\mathscr G(\tilde{E}|k)$ is isomorphic to the factor group $G/\mathrm{ker}.$ Now, we only need to show that $K\subset \tilde{E}.$
	
Since $\mathscr G$ is connected, $\mathscr N$ is contained in the center of $\mathscr G.$ Furthermore, since $\mathscr P$ is parabolic, the normalizer of $\mathscr P$ is itself and we obtain that $\mathrm{ker}(\phi)\subset \mathscr H.$ This shows that $K\subset E^{\mathrm{ker}(\phi)}=\tilde{E}.$\end{proof}

\begin{lemma}\label{semisimple-case}
		Let $E$ be a Picard-Vessiot extension of $k,$ with $\mathscr{G}:=\mathscr G(E|k)$ being a (connected) semisimple algebraic group. Let $k\subset K\subseteq E$ be a differential field. Then, there is an element $\zeta\in T(E|k)\setminus k$ such that $k\subset k\langle \zeta'/\zeta\rangle\subseteq K$ and that $k\subset k\langle \zeta'/\zeta\rangle\subseteq E_1\subseteq E,$ where  $E_1$ is a Picard-Vessiot extension of $k$ whose differential Galois group is a semisimple algebraic group of rank at most $\mathrm{tr.deg}(K|k).$
\end{lemma}

\begin{proof}
	First we claim that  $K^0,$ the algebraic closure of $K$ in $E,$ contains a differential field $M$ such that $M^0,$  the algebraic closure of $M$ in $E,$ is irreducible. To see this, we shall use an induction on $\mathrm{tr.deg}(K| k).$ Clearly, $K^0$ itself is irreducible if $\mathrm{tr.deg}(K| k)=1.$ Suppose that $\mathrm{tr.deg}(K| k)\geq 2$ and  that $K^0$ is reducible. Then we have $k\subset K_1\subset K^0$ such that $1<\mathrm{tr.deg}(K_1|k)<\mathrm{tr.deg}(K^0| k).$ Therefore, by induction, we may assume  $K_1$  has a differential subfield $M$ such that $M^0$ is irreducible. Since $M$ is also a subfield of $K^0,$  we have proved the claim. Thus  there is a differential subfield $M,$ $k\subset M\subseteq M^0\subseteq K^0,$ where $M^0, K^0$ are the algebraic closures of $M$ and $K$ in $E$ respectively, such that $M^0$ is irreducible.  
    
    If $\mathscr H$ is a closed connected subgroup of $\mathscr G$ properly containing $\mathscr G(E|M^0)$ then $E^\mathscr H$ is a differential subfield of $M^0$ such that $\mathrm{tr.deg}(E^\mathscr H|k)<\mathrm{tr.deg}(M^0|k).$ Thus, the irreducibility of $M^0$ along with the Galois correspondence implies that $\mathscr G(E|M^0)$ is maximal among all connected subgroups of $\mathscr G.$ Therefore, $\mathscr G(E|M^0)$ is either a semisimple algebraic group or $\mathscr G(E|M^0)$ is a parabolic subgroup of $\mathscr G$ (\cite[Lemma 2.1]{Lie-Don-04}).
	
	Let $\mathscr H:=\mathscr G(E|M^0).$ Suppose that $\mathscr H$ is a semisimple group. Then  $\mathscr H$ is observable in $\mathscr G$ (\cite[Corollary 4.6]{Gro97}) and therefore $M^0=k\langle y\rangle$ for some $y\in T(M^0|k)\setminus k$ (\cite[Theorem 1.2.2]{And14}). Since $\mathscr G$ is semisimple, there is no nontrivial morphism from $\mathscr G$ to the multiplicative group  $\mathrm{G}_m$ and therefore, $y'/y\notin k.$ Therefore, $k\langle y'/y\rangle\subseteq M^0$ must be an algebraic extension.  Since $\mathscr H$ is reductive, the homogeneous space $\mathscr G/\mathscr H$ is an irreducible affine variety. A result of V.L. Popov (\cite[Corollary 1]{Pop75}), rank of $\mathscr G$ is at most $\mathrm{codim}(\mathscr H)-1$.  Therefore, from Galois correspondence, we obtain that rank of $\mathscr G$ is at most $\mathrm{tr.deg}(M^0|k)-1.$ Finally, since $K^0\supseteq M^0,$ we obtain that $y'/y$ is algebraic over $K.$ Now we apply  Proposition \ref{algebraicriccatisolution-irreduciblefield} and obtain an element $\zeta\in T(E|k)\setminus k$ such that $\zeta'/\zeta\in K\setminus k.$  Since $M^0\subseteq K^0,$ we have $$\mathrm{rank} (\mathscr G)\leq \mathrm{tr.deg}(M^0|k)-1\leq  \mathrm{tr.deg}(K|k).$$
 	
	Suppose that $\mathscr H$ is a parabolic subgroup of $\mathscr G.$  Let $\mathscr G=\mathscr G_1\mathscr G_2\cdots\mathscr G_s$ be the factorization of $\mathscr G$ in terms of its simple components. Then since $\mathscr H$ is a maximal parabolic of $\mathscr G$, there is a simple component , say $\mathscr G_1,$ of $\mathscr G$ such that $\mathscr H$ is a (maximal) parabolic subgroup of $\mathscr G_1$ (\cite[Page 87-88]{Hum78}).  Let $E_1=E^{\mathscr G_2\mathscr G_3\cdots\mathscr G_s}.$ Then $E_1$ is a Picard-Vessiot extension of $k$ whose differential Galois group is $\mathscr G_1$ and since $\mathscr H\subset \mathscr G_1,$ we have $M^0\subset E_1.$ Since $\mathscr G_1$ is simple, it is known that rank of $\mathscr G_1$ is at most $\mathrm{dim}(\mathscr G_1/\mathscr H)$ (\cite[Remark 8.3]{CX18}). Therefore by Galois correspondence, $$\mathrm{rank}(\mathscr G_1)\leq \mathrm{tr.deg}(M^0|k)\leq \mathrm{tr.deg}(K|k).$$
	
	We claim that there is an element $y\in T(E_1|k)$ such that $y\notin k$ and $k\subset k(y'/y)\subseteq M^0.$ 
	Let $V\subset E$ be a finite dimensional $C-$vector space as well as a faithful $\mathscr G_1-$module. Then, from Proposition \ref{reductiontosimplealgebraicgroup}, we shall assume that $V$ is also an irreducible $\mathscr G_1-$module. We have a monomorphism of algebraic groups $\mathscr G_1\hookrightarrow  \mathrm{GL}(V).$ Since parabolic subgroups of  reductive groups are determined by cocharacters (\cite[Proposition 8.4.5]{Sprin-98}), every parabolic subgroup of $\mathscr G_1$ is a pullback of a parabolic subgroup of $\mathrm{GL}(V).$  It is known that parabolic subgroups of $\mathrm{GL}(V)$ stabilizes  partial flags and therefore, there is  a proper nonzero subspace $W$ of $V$ such that $$0\subset W\subset V$$ is a (partial) flag for  $\mathscr H.$  Let $w_1,\dots,w_m$ be a $C-$basis of $W$ and $\omega$ be its Wronskian. Then $\omega\in T(E_1|k).$   Let $\mathcal L_W\in M^0[\partial]$ be the monic differential operator whose full set of solutions is $W.$ Then $$\mathcal L_W=\partial^m+\frac{a_{m-1}}{\omega}\partial^{m-1}+\dots+\frac{a_0}{\omega},$$ where each $a_i\in T(E_1|k).$ Note that  $\omega'/\omega=-a_{m-1}/\omega\in M^0.$ Thus, if $\omega'/\omega\notin k$ then we shall prove the claim by taking $y:=\omega$ and  observing that 
	$$k\subset k\langle y'/y\rangle\subseteq M^0\subseteq E_1\subseteq E, \quad \text{for some } \ y\in T(E_1|k).$$
	
Therefore, we may suppose that $\omega'/\omega\in k.$ Since $k$ is algebraically closed in $E_1$ and that $\mathscr G_1$ is simple, there is no surjective morphism from $\mathscr G_1$ to $\mathrm{G}_m.$ This in turn implies $\omega\in k.$ Since $V$ is irreducible, we know that  $W$ is not a $\mathscr G_1-$module and therefore, there is an $i,$ $0\leq i\leq m-1$ such that $a_i/\omega\notin k.$ That is,  $y:=a_i\notin k.$  Thus, we have  $$k\subset k\langle y'/y\rangle\subseteq k\langle y\rangle\subseteq M^0\subseteq E_1\subseteq E, \quad \text{for some } \ y\in T(E_1|k)$$ and this proves our claim. 

Finally, since $M^0\subseteq K^0,$ we observe that the element $y'/y$ is algebraic over $K$ and that $y$ is not algebraic over $k.$ Now, by  Proposition \ref{algebraicriccatisolution-irreduciblefield}, we obtain a nonzero element $\zeta\in T(E_1|k)$ such that $k\subset k(\zeta'/\zeta)\subseteq K.$
\end{proof}

\begin{remark}\label{productofclassicalgroups} Let $E$ be a Picard-Vessiot extension of $k$ and $\mathscr G:=\mathscr G(E|k)$ be a semisimple algebraic group such that all its  components are classical simple algebraic groups.  Let $K$ be a $k-$irreducible differential subfield of $E.$ From Lemma \ref{semisimple-case}, it follows  that the $r:=\mathrm{rank}(\mathscr G)\leq \mathrm{tr.deg}(K|k).$ Let $\mathscr G=\mathscr G_1\cdots \mathscr G_n$ be the decomposition of $\mathscr G$ into simple components and $r_i:=\mathrm{rank}(\mathscr G_i).$ Since $\mathscr G_i$ are assumed to be classical simple groups, $\mathrm{dim}(\mathscr G_i)\leq (r_i+1)^2-1$ for each $i=1,2,\dots,n$ and since $\sum^n_{i=1}r_i=r,$ we obtain that $$\mathrm{tr.deg}(E|k)=\mathrm{dim}(\mathscr G)\leq -n+\sum^n_{i=1}(r_i+1)^2 \leq (r+1)^2-1=r(r+2).$$ This in turn gives us bounds for the transcendence degree  of $k-$irreducible differential subfields of $E$ in terms of $\mathrm{rank}(\mathscr G):$
\begin{equation}\label{bound-r(r+2)}r\leq \mathrm{tr.deg}(K|k)\leq r(r+2).\end{equation}
\end{remark}

Now we are in a position to prove Theorem \ref{intromain}.

\begin{proof}[Proof of Theorem \ref{intromain}] If $K$ is an algebraic extension of $k$ then the theorem is proved by taking $y=1.$ Therefore, we shall assume that $\mathrm{tr.deg}(K|k)\geq 1.$ If $k_1$ is the algebraic closure of $k$ in $K_1$ then since $T(K|k_1)=T(K|k),$ we may assume that $k_1=k,$ that is $k$ algebraically closed in $K.$
	
Let $\mathscr R$ and $\mathscr U$ be the radical and the unipotent radical of $\mathscr G(E|k).$	First we observe that for any connected closed subgroup $\mathscr H$ of $\mathscr G(E|k),$ if  $E^\mathscr H\cap K\neq k$ then since $K$ is irreducible, the field $K$ must be an algebraic extension of $E^\mathscr H\cap K.$ Now since $\mathscr H$ is connected, we obtain that  $K\subseteq E^\mathscr H.$  Taking $\mathscr H$ to be $\mathscr R$ and $\mathscr U,$ we arrive at the following possibilities:
\begin{enumerate}[(a)]\item $K\cap E^\mathscr U=k,$\\
	\item  $K\subseteq E^{U}$ and $K\cap E^{\mathscr R}=k,$\\
\item $K\subseteq E^\mathscr R.$\end{enumerate}

Suppose that $K\cap E^\mathscr U=k.$ Then, from  Lemma \ref{unipotent-case}, we obtain a element $\zeta\in T(K|k)$ not algebraic over $k$ and a monic $\mathcal L\in k[\partial]$ such that $\mathcal L(\zeta)=b\in k\setminus \{0\}.$ Now since $\zeta$ is transcendental over $k,$ it is easily seen that $\zeta'/\zeta\notin k$ and we obtain that  $K$ is an algebraic extension of 
$k\langle \zeta'/\zeta\rangle.$ Now suppose that $K\subseteq E^\mathscr U$ and $K\cap E^{\mathscr R}=k.$  Then since $\mathscr G(KE^\mathscr R|K)\cong \mathscr G(E^\mathscr R|k),$ we obtain $E^{\mathscr R}$ and $K$ are free over $k.$ Now, the assumption that $k$ is algebraically closed in $K$ implies that $K$ and $E^{\mathscr R}$ are linearly disjoint over $k.$ This implies, $k^0K$ and $E^{\mathscr R}$ are linearly disjoint over $k^0$ and in particular, $k^0K\cap E^{\mathscr R}=k^0.$ Apply Lemma \ref{torus-case} and obtain an element $y\in k^0K\setminus k^0$ such that $y'/y\in k^0.$  Observe that $y\in T(k^0K|k^0)$ and from Proposition \ref{solnoveralb.extn} that $T(k^0K|k^0)=T(k^0K|k).$ Therefore, $y\in T(k^0K|k)\subseteq T(E^\mathscr R| k)$  and that $y$ is algebraic over $K.$ Now from Proposition \ref{algebraicsolution-irreduciblefield}, we obtain an element $\zeta\in T(K|k)\setminus k.$ If $\zeta'/\zeta\notin k$ then $K$ is an algebraic extension of $k\langle \zeta'/\zeta\rangle$ and if $\zeta'/\zeta\in k$ then we simply replace $\zeta$ by $\zeta+\alpha$ for some nonzero $\alpha\in k$ and observe that $\zeta\in T(K|k)\setminus k,$ $\zeta'/\zeta\notin k$ and that  $K$ is an algebraic extension of the differential field $k(\zeta'/\zeta).$ 

Finally, suppose that $K\subseteq E^{\mathscr R}.$ Apply Lemma \ref{semisimple-case} for $k^0\subseteq k^0K\subseteq E^\mathscr R$ and obtain a Picard-Vessiot subextension $E_1$ of $k^0$ and a nonzero element $y\in T(E_1|k^0)$ such that $$k^0\subset k^0\langle y'/y\rangle\subseteq k^0K\ \text{and that }\  k^0\subset k^0\langle y'/y\rangle \subseteq E_1.$$  By Proposition \ref{solnoveralb.extn}, $T(E_1|k^0)=T(E_1|k).$ Thus $y\in T(E_1|k),$  $y$ is not algebraic over $k$ and that $y'/y$ is algebraic over $K.$  Now, from Proposition \ref{algebraicriccatisolution-irreduciblefield}, we obtain a nonzero element $\zeta\in T(E_1|k)$ such that $\zeta'/\zeta\in K\setminus k.$ Since $K$ is $k-$irreducible, it must be an algebraic extension of  $k\langle \zeta'/\zeta\rangle.$ This completes the proof of the theorem.\end{proof}

\begin{remark}\label{obv-intromain}
    Let $E$ be a Picard-Vessiot extension of $k$ with a connected Galois group $\mathscr G (E|k)$. Let $K$ be a relatively algebraically closed $k-$irreducible differential subfield of $E$ and $\mathrm{tr. deg}(K|k)=n.$ Then,  from the proof of Theorem \ref{intromain}, we observe that one of the following statements hold:
    \begin{enumerate}[(i)]
        \item \label{from-unipotent} There is an element $y\in K$ such that $\mathcal L(y)=b$ for some  $\mathcal L\in k[\partial]$ and $b\in k$  with $\mathrm{ord}(\mathcal L)=n$ and that  $K$ is an algebraic extension of $k\langle y\rangle.$
        \item\label{from-torus} $\mathrm{tr. deg}(K|k)=1$ and there is an nonzero element $z\in K$ such that $K$ is an algebraic extension of $k(z)$ with  $z'/z\in k.$\\
        \item \label{from-semisimple} $K$ is a differential subfield of a Picard-Vessiot extension $F$ of $k$ with $\mathscr G(F|k)$ a semisimple linear algebraic group of rank $\leq n-1$  and $K=F^\mathscr{H}$ for a maximal semisimple subgroup $\mathscr H$ of $\mathscr G(F|k).$ \\
        \item  \label{from-parabolic} $K$ is a differential subfield of a Picard-Vessiot extension $L$ of $k$ with $\mathscr G(L|k)$ a simple linear algebraic group of rank $\leq n$ and that $K=L^\mathscr{P}$  for a maximal parabolic subgroup $\mathscr P$ of $\mathscr G(L|k).$ 
    \end{enumerate} 
\end{remark}

\begin{remark}
    
Let $E$ be a Picard-Vessiot extension of $k$ and let $\mathscr G$ be the  differential Galois group of $E$ over  $k$. Then from  Theorem \ref{intromain}, every differential field intermediate to $k$ and $E$ contains  a $k-$irreducible differential field $K:=k\langle w\rangle,$ where $w$ is a solution of a Riccati differential equation  $\mathcal R(w,w',\dots,w^{(n-1)})=0.$  For any $\sigma\in \mathscr G,$ the differential field $\sigma(K)$ is  irreducible  and the field compositum $\prod_{\sigma \in \mathscr G}\sigma(K)$  is stable under the action of $\mathscr G,$ hence a Picard-Vessiot extension of $k.$ Moreover, $\prod_{\sigma \in \mathscr G}\sigma(K)$ is an algebraic extension of $k\langle w_1,\dots, w_m\rangle,$ where $w_1, \dots, w_m$ are $k-$algebraically independent and $\mathcal R(w_i,w'_i,\dots, w^{n-1}_i)=0$ for each $i=1,2\dots,m.$ In particular, if $\mathscr G$ is a simple algebraic group then $E$ is a finite algebraic extension of  $k\langle w_1,\dots, w_m\rangle$ and the differential Galois group of  $\prod_{\sigma \in \mathscr G}\sigma(K)$ is isogenus to $\mathscr G.$ This structure of $E,$ when  $\mathscr  G$ is isogeneous to $\mathrm{SL}_2,$ is well-known (\cite[pp. 57-58]{Mag94}). 
\end{remark}


\section{Differential subfields of strongly normal extensions}

In this section, we will prove Theorem \ref{intromain-SN}.  The proof will use Chevalley's  theorem on the structure of algebraic groups (\cite[Theorem 16 \& Corollary 3]{Ros-56}) and  results on  Albanese variety of a homogeneous space (\cite[Lemma 1 and its Corollary, Theorems 2 \& 3]{Ish65}, \cite[Theorem A]{JSC22}).

\begin{proof}[Proof of Theorem \ref{intromain-SN} (\ref{intromain-irred})]  
Let $\mathscr G:=\mathscr G(E|k),$  $\mathscr{N}$ be the maximal connected  normal linear algebraic subgroup of $\mathscr{G}^0$ such that the factor group $\mathscr{G}^0/\mathscr{N}$ is an abelian variety and $\mathscr{D}$ be the smallest normal algebraic subgroup of $\mathscr{G}$ such that $\mathscr{G}/\mathscr{D}$ is a linear algebraic group. Note that $\mathscr D\subseteq \mathscr G^0$ and that $\mathscr G^0/\mathscr D$ is also a linear algebraic group. By the Galois correspondence of strongly normal extensions, $E^{\mathscr{N}}$ is an abelian extension of $k^0=E^{\mathscr G^0}$ and $E^{\mathscr{D}}$ is a Picard-Vessiot extension of $k.$  We shall prove that either $k^0K$ is an abelian extension of $k^0$ or $K\subseteq E^{\mathscr D}.$
	
	Suppose  that $k^0=k,$ that is $\mathscr G$ is a connected group.  Let $\mathscr{H}=\mathscr{G}(E|K)$ and observe that $K=E^{\mathscr{H}}.$ Consider the homogeneous space $V:=\mathscr{G}/\mathscr{H}.$ Then, $V$ is a nonsingular quasi-projective variety over $C.$ Let $\mathrm{Alb}(V)$ be an Albanese variety of $V.$ Then from \cite{Ish65},  we obtain that  \begin{enumerate}[(i)]
		\item $\mathrm{dim}\  \mathrm{Alb}(V)\leq \mathrm{dim}\ V,$
		\item $\mathrm{dim}\ \mathrm{Alb}(V)=0$ if and only if the isotropy group of any point on $V$ contains $\mathscr{D},$
		\item $\mathrm{dim}\ \mathrm{Alb}(V)=\mathrm{dim}\ V$ if and only if the isotropy group of any point on $V$ contains $\mathscr{N},$
		\item Given a homogeneous space $V$ of $\mathscr{G},$ there exists a homogeneous space $V_{\mathscr{N}}$ of the abelian variety $\mathscr{G}/\mathscr{N}$ of dimension $\mathrm{dim} \  \mathscr{G}-\mathrm{dim} \ \mathscr{N}\mathscr{H}$ such that  $V_{\mathscr{N}}$ is  an Albanese variety of $V.$    
	\end{enumerate}
	Note that $\mathscr{H}$ is an isotropy group of a point (namely $\mathscr H$) of $V$ and that $\mathrm{dim}\ V= \mathrm{Codim}\  \mathscr{H}=\mathrm{ tr. deg}(K/k).$ We now split the rest of the proof into two cases.
	
	Suppose that  $\mathrm{dim}\ \mathrm{Alb}(V)=0,$ then $\mathscr{D}\subseteq \mathscr{H}$ implies $K=E^{\mathscr{H}}\subseteq E^{\mathscr{D}}$ and therefore, $K$ must be a differential subfield of a Picard-Vessiot extension.
	
	Suppose that $\mathrm{dim}(\mathrm{Alb}(V))\geq1.$ Then $\mathrm{dim}(\mathscr{G})-\mathrm{dim}(\mathscr{N}\mathscr{H})\geq 1$ and therefore $\mathscr{N}\mathscr{H}$ is a proper subgroup of $\mathscr{G}$ and $E^{\mathscr{N}\mathscr{H}}\neq k.$ Since $\mathscr{H}\subsetneq\mathscr{N}\mathscr{H}\subsetneq\mathscr{G},$ by Galois correspondence, $k\subsetneq E^{\mathscr{N}\mathscr{H}}\subseteq E^{\mathscr{H}}=K.$ By assumption, $K$ is irreducible and therefore we have $K=E^{\mathscr{N}\mathscr{H}}(\alpha),$ where $\alpha$ is algebraic over $E^{\mathscr{N}\mathscr{H}}.$ We also have the following containment of closed subgroups $\mathscr{N}\subseteq\mathscr{N}\mathscr{H}\subsetneq\mathscr{G},$ which in turn gives the following containment of differential fields  $k\subsetneq E^{\mathscr{N}\mathscr{H}}\subseteq E^{\mathscr{N}}.$ Since $\alpha$ is algebraic over $E^{\mathscr{N}\mathscr{H}},$ it is also algebraic over $E^\mathscr{N}.$ Now since $\mathscr{N}$ is connected, $E^\mathscr{N}$ must be relatively algebraically closed and we obtain that  $\alpha \in E^{\mathscr{N}}.$ Therefore, $K=E^{\mathscr{N}\mathscr{H}}(\alpha)\subseteq E^\mathscr{N}.$ Thus $K,$ being a differential subfield of an abelian extension of $k,$  is an abelian extension of $k$. This completes the proof of the theorem when $\mathscr G$ is connected.
	
	Next, we suppose that  $k^0\neq k.$ Observe that $k\subseteq k^0\subseteq E^\mathcal D.$  Since $k^0=E^{\mathscr G^0}$  and that  $E$ is strongly normal extension of $k^0$ with a connected differential Galois group $\mathscr G^0,$  we obtain  that either $k^0K$ is  a subfield of the Picard-Vessiot extension  $k^0E^\mathcal D=E^\mathcal D$ of $k$ or $k^0K$ is an abelian extension of $k^0.$ 
\end{proof}

\begin{remark} \label{remarksonintromainSN}
 If $C\subset K\subset E$ are differential fields such that $E$ is a strongly normal extension of $C$ and $K$ is $C-$irreducible then $K$ is either an abelian extension of $C$ or $K$ can be embedded in a Picard-Vessiot extension of $C.$ In the latter case, since the differential Galois group of any Picard-Vessiot extension of $C$ is connected, commutative and isomorphic to $\prod^r \mathrm G_m$,  $\mathrm G_a$ or $\mathrm G_a \times \prod^r \mathrm G_m,$ we obtain that $K$ itself must be a Picard-Vessiot extension of $C.$ Furthermore, since $K$ is $C-$irreducible, there is a nonzero element $x\in K$ such that $K=C(x)$ and either  $x'=1$ or $x'/x\in C\setminus \{0\}.$ 
	\end{remark}


\begin{proof}[Proof of Theorem \ref{intromain}(\ref{intromain-uni})]
	Let $\mathscr G:=\mathscr G(E|k)$ be the differential galois group of $E$ over $k.$ First we suppose that $k$ be an algebraically closed field. It is known that $K,$ as a field, is isomorphic to the field  of fractions of the domain $k\otimes_C C(\mathscr G/ \mathscr H)),$ where $\mathscr H=\mathscr G(E|K)$ (\cite[Theorem 1]{BB62} or \cite[Chapter I, Section 3]{Bui86}).  Let $V=\mathscr G/\mathscr H$ and $V_k:= k\times_C V$.  Since $K$ is a subfield of a purely transcendental extension field of $k,$ we have the following morphisms $$\mathbb P^n(k)\xrightarrow{\phi} V_k\xrightarrow{\mathrm{Alb}_{V_k}} \mathrm{Alb}(V_k),$$ where $\mathrm{Alb}_{V_k}$ is the Albanese morphism and  $\phi$ is the dominant morphism induced by the inclusion map. It is known that there is no nonconstant morphism from an unirational variety to any abelian variety ( \cite[Corollary of Theorem 4]{Lang-59}) and therefore the composition map must be trivial. Since $\phi$ is dominant and that $V_k$ is irreducible, it follows that $\mathrm{Alb}_{V_k}$  is also trivial and thus $ \mathrm{dim_k \ Alb}(V_k)=0.$ It is been recently proved that the construction of an Albanese variety of a given variety respects base change and thus $\mathrm{dim_C\ Alb}(V)=\mathrm{dim_k \ Alb}(V_k)=0$ (\cite[Theorem (A)]{JSC22}). This implies $\mathrm{Alb}(V)$ is trivial. 
	
	Let $\mathscr N=\{g\in \mathscr G\ | \ gv=v\ \text{for all }\ v\in V\}.$ Then it is easily seen that $\mathscr N\subseteq \mathscr H,$  $\mathscr N$ is  normal subgroup of $\mathscr G$ and that the induced action of $\mathscr G/\mathscr N$ on $V$ is faithful. Therefore, the kernal of the induced map $\mathscr G/\mathscr N\to \mathrm{Alb}(V)=\{e\}$ must be a linear algebraic group (\cite[Theorem 2]{Matsumura-63} or \cite{Brion-2010}). That is, $\mathscr G/\mathscr N$ must be a linear algebraic group. Thus, we obtain that $E^\mathscr N$ is a Picard-Vessiot extension of $k$ and that $k\subset K\subseteq E^\mathscr N$.

    Now, let $\overline{k}$ be the algebraic closure of $k.$  From \cite[Theorem 5]{Kol53}, $\overline{k}E$ is a strongly normal extension of $\overline{k}$ and the morphism $$\pi:\mathscr G(\overline{k}E| \overline{k})\rightarrow \mathscr G(E| \overline{k}\cap E)=\mathscr G(E|k^0),\text{ } \text{$\sigma \mapsto \sigma|_E$}$$ is an isomorphism of algebraic groups. We know from the above discussion that there exists a normal algebraic subgroup $\mathscr N\subset \mathscr G(\overline{k}E| \overline{k})$ such that $\overline{k}K\subseteq (\overline{k}E)^\mathscr N$ and that $\mathscr G(\overline{k}E| \overline{k})/\mathscr N$ is a linear algebraic group.  Consider the algebraic group $\pi(\mathscr N).$ Note that $\mathscr G(\overline{k}E| \overline{k})/\mathscr N\cong\mathscr G(E/k^0)/\pi(\mathscr N)$ implies $\mathscr G(E/k^0)/\pi(\mathscr N)$  is a linear algebraic group. Therefore $E^{\pi(\mathscr N)}$ is a Picard-Vessiot extension of $k^0.$ Since $\overline{k}K\subseteq (\overline{k}E)^\mathscr N,$ it follows that $K\subseteq E^{\pi(\mathscr N)}.$ As $k^0$ is a finite Galois extension of $k,$ the differential field $E^{\pi(\mathscr N)}$ can be embedded in a Picard-Vessiot extension $L$ of $k$ such that $L\subseteq E$ (\cite[Proposition 11.4]{Kov06}). Thus $k\subset K\subseteq E^{\pi(\mathscr N)}\subseteq L$ and this completes the proof.   \end{proof}

\subsection{Differential subfields of iterated strongly normal and iterated  Picard-Vessiot extensions}

A finitely generated differential field extension $K$ of $k$ is said to be \emph{depend rationally on arbitrary constants over} $k$ if there is a differential field extension $M$ of $K$ satisfying the following conditions.
\begin{enumerate}[(a)] \item   $K$ and $M$ are free over $k$.\\
	\item $KM=MC_{KM},$ where $C_{KM}$ is the field of constants of the compositum $KM.$
\end{enumerate}   
In a series of works (see \cite{BB62, NK88, NK89, NK89a, NK94, HU90, Hu97}), A. Bialynicki-Birula, K. Nishioka and H. Umemura studied differential fields depend rationally on arbitrary constants and obtained various results that enabled to study Painlev\'{e} equations and strongly normal extensions. Here, we shall consolidate their results that are relevant  to our work and  prove a structure result for differential subfields of iterated strongly normal extensions.

Let $K$ and $M$ be differential fields such that $K\subseteq M\subseteq KC_E.$ Then it is shown in \cite[Lemma]{NK89a} that $M=KD,$ for some field  $D\subset C_E$. Using this result, it can be shown that if a differential field extension $E$ of $k$ depends rationally on arbitrary constants over $k$ then its differential subfields also depend rationally on arbitrary constants over $k.$  Let $E$ be a Picard-Vessiot extension of $k$ for an operator $\mathcal L:=\partial^n-\sum_{i=0}^{n-1}a_i\partial^i\in k\left[\partial\right]$ with full set of solutions $V$ with basis $\{v_1,\dots,v_n\}.$   Consider the rational function field $L=E(X_{ij}| 1\leq i,j\leq n )$ with the derivation $X'_{ij}=X_{i+1j}$ for $i<n-1$ and $X'_{n-1j}=\sum_{i=0}^{n-1}a_iX_{ij}.$ By definition, $M:=k(X_{ij})$ and $E$ are free over $k.$ It is easy to  see that $(v_{j}^{(i)})=(X_{ij})(c_{ij}),$ where $(c_{ij})\in \mathrm{GL}_n(C_{EM}).$ Therefore, $EM=MC_{EM}$ and we obtain the Picard-Vessiot extensions depend rationally on arbitrary constants. If $K$ is an algebraic extension of $k$ and $E$ is a finite  Galois extension of $k$ containing $K$ then since $E$ is also a Picard-Vessiot extension of $k$, it follows that $K$ depends rationally on arbitrary constants.  

More generally, if $E$ be a strongly normal extension of $k$, then $E$ depends rationally on arbitrary constants. To see this, let $k^0$ be the algebraic closure of $k$ in $E$.  Then $E\otimes_{k^0} E$ is a domain and we have the differential embeddings  $i_1: E\to E\otimes_{k^0} E,$ $i_1(e)=e\otimes 1$ and $i_2: E\to E\otimes_{k^0} E,$ $i_2(e)=1\otimes e.$ Let $\mathcal E$ be the field of fractions of $(E_1\otimes_{k^0} E_2).$ Clearly, $i_1(E)$ and $i_2(E)$ are linearly disjoint over $k^0$ and consequently, they are free over $k.$ Now since $i_1$ and $i_2$ are differential embeddings and $E$ is a strongly normal  extension of $k,$ we obtain that  $i_1(E)i_2(E)= \mathcal E=i_2(E)C_{\mathcal E}.$ Thus, any strongly normal extension $E$ of $k$ depends rationally on arbitrary constants of $k.$   

If a differential field extension $K$ of $k$ can be resolved into a tower of differential fields
$$k=K_0\subseteq K_1\subseteq K_2\subseteq \cdots \subseteq K_n=K,$$ where each $K_i$ depend rationally on arbitrary constants the $K$ is called a \emph{Painlev\'{e}-Umemura extension} of $k.$
Thus, if $E$ is an iterated strongly normal extension of $k$ then $E$ is a  Painlev\'{e}-Umemura extension of $k.$

\begin{proof}[Proof of Theorem \ref{iter-SN-Subfields}]
	We first prove the theorem when $K$ is $k-$irreducible. Let $k^0$ be the algebraic closure of $k$ in $K$. Then, since $E$ is an iterated strongly normal extension of $k^0,$ we obtain from \cite[Theorem]{NK94} that $K$ is  a Painlev\'{e}-Umemura extension of $k^0$ and that 
	$$k^0\subset M\subseteq K,$$ where $M$ is a differential subfield of a strongly normal extension of $k^0$ (see \cite[Theorem 23]{HU90}). Now, from Remark \ref{embeddingSNafterfinitalg}(\ref{SNafterfinitealg}), $M$ is actually a differential subfield of a strongly normal extension of $k.$ Since $K$ is $k-$irreducible,  $K$ must be an algebraic extension of $M.$ Now suppose that $E$ is an iterated Picard-Vessiot extension of $k.$ Then from Theorem \ref{intromain-SN}, either $M$ is a differential subfield of a Picard-Vessiot extension of $k$ or  
 there is an algebraic extension $\tilde{k}$ of $K$ such that $\tilde{k}M$ is an abelian extension of $\tilde{k}.$ We shall show that  the latter situation does not occur. 
	
Assume, on the contrary, that $\tilde{M}:=\tilde{k}M$ is an abelian extension of $\tilde{k}.$ We know that $\tilde{E}:=\tilde{k}E$ is an iterated Picard-Vessiot extension of $\tilde{k}.$ Let
	$$\tilde{k}=E_0\subseteq E_1\subseteq\cdots \subseteq E_r=\tilde{E},$$ where each $E_i$ is a Picard-Vessiot extension of $E_{i-1}.$  Without loss of generality, we shall assume that $\tilde{M}\not\subset E_{r-1}^0$ and $\mathrm{tr. deg}(\tilde{M}E_{r-1}^0|E_{r-1}^0)\geq 1$ where $E_{r-1}^0$ is algebraic closure of $E_{r-1}$ in $E_r.$  From \cite[Theorem 5]{Kol73}, we have $$\mathscr G(\tilde{M}E_{r-1}^0|E_{r-1}^0)\cong \mathscr G(\tilde{M}| \tilde{M}\cap E_{r-1}^0).$$ Note that $\mathscr G(\tilde{M}|\tilde{M}\cap E_{r-1}^0),$ being a closed subgroup of an abelian variety,  is an abelian variety. Therefore, $\mathscr G(\tilde{M}E_{r-1}^0|E_{r-1}^0)$ is an abelian variety and $\tilde{M}E_{r-1}^0$ is an abelian extension of $E_{r-1}^0$.  This implies $\mathscr G(\tilde{E}|\tilde{M}E_{r-1})$ is a normal subgroup of the connected linear algebraic group $\mathscr G(\tilde{E}|E_{r-1}^0)$ such that the corresponding quotient group is the abelian variety $\mathscr G(\tilde{M}E_{r-1}^0|E_{r-1}^0).$ This is impossible as there is no nonconstant morphism from a connected linear algebraic group to an abelian variety. Thus $\tilde{M}$ cannot be an abelian extension of $\tilde{k}$, and this completes the proof of the theorem when $K$ is $k-$irreducible. 
	
	Now we only need to consider the case when  $\mathrm{tr.deg}(K|k)\geq 2$ and that $K$ is $k-$reducible. In this case, we have $$k\subset M\subset K,$$  where $M$ is a differential field such that both $\mathrm{tr.deg}(M|k)$ and $\mathrm{tr.deg}(K|M)$ are less than $\mathrm{tr.deg}(K|k).$ Since $E$ is also an iterated strongly normal extension (respectively, iterated Picard-Vessiot extension)of $M,$  through an induction on the transcendence degree, we obtain the required tower of differential fields for $K.$
\end{proof}

\subsection{An extension of Rosenlicht's result} In Proposition \ref{strong-ros-extn}, we shall establish an intricate connection between the existence of new constants after base change and the presence of differential subfields that can be embedded in  strongly normal extensions.  As we will see, the proposition is proved merely by applying existing results from the theory of differential fields algebraically dependent on arbitrary constants. Before we proceed to state the proposition, we first recall few notations and definitions from \cite{NK89a}. 

For two multi-indices\footnote{A sequence $J=(j_n)$, where $j_n$s are non-negative integers and $j_n=0$ for all but finitely many is called a \emph{multi-index}.} $I=(i_n)$ and $J=(j_n),$ defined $I<J$ if
there is an integer $m$ with $i_m<j_m$ and $i_n=j_n$ for all $n\geq m.$ Let $k\left\lbrace Y\right\rbrace $ denote the algebra of differential polynomials with coefficients in $k.$ An element $A$ of $k\left\lbrace Y\right\rbrace $ can be written as $A=\sum_{J}a_J Y^J,$ where $J$ runs through all multi-indices,  $a_J$ belongs to  $k$ and $Y^J=\prod_{n\geq 0}Y_n^{j_n},$ where $Y_n$ is  the $n-$th derivative of $Y.$ The rank of the non-zero differential polynomial $A$ is defined as the highest multi-index $J$ with $a_J\neq0.$ Let $x$ be an element of a  finitely generated differential field extension of $k.$ Then, there exists a nonzero differential polynomial $A$ over $k$ such that $A(x)=0.$ The rank of $x$ over $k,$ denoted by $\mathrm{rank}_k(x),$ is defined as the lowest among the ranks of all differential polynomials over $k$ that vanish at $x.$  Define $$r(x):=\min\left\{ \mathrm{rank}_{MC_{M(x)}}(x)\ | \ M \ \text{and}\ k(x)\ \text{are linearly disjoint over }\ k\right\}.$$ The following observations are readily seen.  

\begin{enumerate}[(i)]
	\item $r(x)<\mathrm{rank}_k(x)$\\  
	\item If $r(x)=(j_n)$ and $j_0=1$ and $j_n=0$ for all $n>0,$ then the differential field $k\left\langle x\right\rangle $ depends rationally on arbitrary constants and if $j_n=0$ for all $n>0,$ then the differential field $k\left\langle x\right\rangle $ depends algebraically on arbitrary constants. \\
	\item If $\mathrm{tr.deg}(k\langle x\rangle|k)=m$ then $\mathrm{rank}_k(x)=(j_n),$ where $j_m\neq 0$ and $j_n=0$ for all $n\geq m+1.$\\
	\item If $\mathrm{rank}_k(x)=(j_n),$ where $j_m\neq 0$ and $j_n=0$ for all $n\geq m+1$ then $\mathrm{tr.deg}(k\langle x\rangle|k)\leq m$
\end{enumerate}

\blfootnote{It is worth noting that Umemura (\cite{HU90}) proves that any soluton of the first differential equation of Painlevé: $y'=6y^2+x$ is "essentially transcendental" by showing that $r(y)=(0,0,1,0,\dots).$}

\begin{proposition}\label{strong-ros-extn} Let $K, M$ be finitely generated differential subfields of some differential field extension of $k,$  $K$ and $M$ be free over $k.$  If $\mathrm{tr.deg}(C_{KM}|C_M)=l\geq 1$ then  either $C_K\neq C$  or $K$ has a differential subfield $L$ with $\mathrm{tr.deg}(L|k)\geq l$ that can be embedded in a strongly normal extension of $k.$  \end{proposition}

\begin{proof}  We may assume that $C_K=C$ and in view of Remark \ref{embeddingSNafterfinitalg}(\ref{SNafterfinitealg}), that $k$ is algebraically closed in $K.$   Suppose that $K=k\langle y\rangle$ and that $\mathrm{tr.deg}(K|k)=m.$ Then from \cite[Theorem]{NK89a}, there is a differential field $N,$ intermediate to $k$ and  $K$ that depends rationally on arbitrary constants and that $N$ is maximal among all such differential fields and that $r(y)=\mathrm{rank}_N(y).$ Note that $r(y)\leq \mathrm{rank}_{MC_{KM}}(y)$ and since $\mathrm{tr.deg}(C_{KM}|C_M)=l,$  we have $\mathrm{rank}_{MC_{KM}}(y)=(j_n),$ where $j_n=0$ for all $n\geq m-l+1.$ Thus,
	$$r(y)=\mathrm{rank}_N(y)\leq (j_n).$$
As noted earlier, we obtain that $\mathrm{tr.deg}(N|k)\geq l.$ Since $C_K=C$ has no constants, from \cite[Theorem]{NK89}, it follows that $N$ can be embedded in a strongly normal extension of $k$. In the general case when $K\neq C,$ one can always find an element $y\in K$ such that $K=k\langle y\rangle$  and thus the proof of the Proposition is complete. However,  if $k=C$ then we need the following approach detailed in \cite[Corollary 42.1]{HU90}. Adjoin a variable $x$ to $C$ and extend the zero derivation on $C$ to a derivation on $C(x)$ with $x'=1.$ Now, we have $K(x)=C(x)\langle y\rangle$ and in this case we know the existence of a differential subfield $N_1$ of $K(x)$ that depends rationally on arbitrary constants. Such a subfield  $N_1$ is known to be a compositum of $C(x)$ and a subfield $N$ of $K$ that depends rationally on arbitrary constants. Now we apply \cite[Theorem]{NK89} to obtain the desired result. \end{proof}

	\begin{remark}(cf. \cite[Proposition]{Ros74}) \label{ros-extn} Let $K, M$ be finitely generated differential subfields of some differential field extension of $k,$  $C_K=C$, $K$ and $M$ be free over $k$ and  $k$ be algebraically closed in $K.$  Suppose that $C_{KM}\neq C_{M}.$ Then since $k$ is algebraically closed in $K$, we have $M$ algebraically closed in $KM.$ Therefore $\mathrm{tr.deg}(C_{KM}|C_M)\geq 1.$ Now from the above proposition and from Theorems  \ref{intromain}, \ref{intromain-SN}, we obtain a $k-$irreducible subfield $L$ of $K$ that can be embedded in a strongly normal extension of $k$ and that either $\tilde{k}L$ is an abelian extension of an algebraic extension $\tilde{k}$ of $k$ or $L=k\langle w\rangle,$ where $w$ is a nonalgebraic solution of a Riccati differential equation. Furthermore, if  $k=C$ then as noted in Remark \ref{remarksonintromainSN}, $L=C(x),$ where either $x'=1$ or $x'/x=c\in  C\setminus \{0\}.$
\end{remark}

\begin{remark}\label{EJ-result}
    We shall now briefly illustrate a special case of  Proposition \ref{strong-ros-extn}, which has already been established by Eagles \& Jimmenez \cite[Corollary]{eagles-Jimenez} very recently.  Let $M$ be a differential field extension of $C$ and $M(x_1,\dots,x_n)$ be a rational function field in $n$ variables. For $f_1,\dots,f_n\in C(x_1,\dots,x_n),$ extend the derivation on $M$ to $M(x_1,\dots,x_n)$ by defining $x'_i=f_i(x_1,\dots,x_n),$ for $i=1,2,\dots, n.$ Clearly, $M$ and $C(x_1,\dots,x_n)$ are free over $C.$ If the differential field $M(x_1,\dots,x_n)$ has a constant which is not in $M$ then from Theorem \ref{intromain-SN}(\ref{intromain-uni}) and the Remark \ref{ros-extn}, we obtain an element  $z\in C(x_1,\dots,x_n)\setminus C$  such that  either $z'=0,$ $z'=1$ or  $z'=cz$ for some $c\in C\setminus 0.$ These conclusions can be rewritten as a follows: There is an element  $z\in C(x_1,\dots,x_n)\setminus C$ such that $$\sum^n_{i=1}f_i(x_1,\dots,x_n)\frac{\partial z}{\partial x_i}=\begin{cases}0\\1\\ cz\end{cases}.$$ \end{remark}



\section{Differential algebraic dependency of solutions of algebraic differential equations} In this section, $k$ is assumed to be an algebraically closed differential field and $\mathfrak U$ is an universal extension of $k.$    Let $f\in k[Y_1,\dots,Y_{n+1}],$ $n\geq 1,$ be an irreducible polynomial with $\partial f/\partial Y_{n+1}\neq 0$ and $y_1\in \mathfrak U$ be a nonalgebraic solution of the differential equation $f(y,y',\dots,y^{(n)})=0.$

\begin{proposition}\label{algindp-auto} 
	If  $k\langle y_1\rangle$ contains a strongly normal extension of $k$ then $\mathfrak S_{k,f,y_1}=\{y_1\}.$  
\end{proposition}

\begin{proof}
	Suppose that $k\langle y_1 \rangle$ contains a strongly normal extension $K$ of $k$ and $ y_2\in \mathfrak S_{k,f,y_1}.$  Then,  any  differential $k-$isomorphism $\sigma:k\langle y_1\rangle\to  k\langle y_2\rangle$  restricts to a differential $k-$automorphism of $K.$ That is, $k\subset K\subseteq k\langle y_1\rangle\cap k\langle y_2\rangle.$ This implies $\mathrm{tr.deg}(k\langle y_1,y_2\rangle)<\mathrm{tr.deg}(k\langle y_1\rangle+$ $\mathrm{tr.deg}(k\langle y_2\rangle.$  Thus $y_1=y_2.$  
\end{proof}

\begin{proof}[Proof of Theorem \ref{algebraicdependence-aode}(\ref{nongeneral-SN})] Suppose that $\mathfrak S_{k,f,y_1}$ is a finite set having $m$ elements. Let $M=\mathrm{Frac}\left(\bigotimes^{m} k\langle y_1\rangle\right)$ and $E=\mathrm{Frac}\left(M\otimes k\langle y_1\rangle\right).$  We may assume that $C_M=C.$ Note that $E$  contains the following $m+1$  elements  $$\tilde{y}_1:=\frac{y_1\otimes 1\otimes 1\cdots\otimes 1}{1},\dots,\tilde{y}_{m+1}:= \frac{1\otimes1\otimes\cdots\otimes 1\otimes y_1}{1},$$ where
	 for each $i=1,2,\dots, m,$ $f(\tilde{y}_i,\tilde{y}'_i,\dots, \tilde{y}^{(n)}_i)=0$ and that $\mathrm{tr.deg}(E|k)=(m+1)\ \mathrm{tr.deg}(k\langle y_1\rangle|k)$.  Since $\mathfrak U$ is universal extension of $k,$  there is a $k-$embedding of $E$ into $\mathfrak U$ and since $\mathfrak S_{k,f,y_1}$ contains only $m$ elements, we must therefore have $C_{E}\supset C_M=C.$ Clearly, $M$ is algebraically closed in $E$ and therefore  $\mathrm{tr.deg}(C_{E}|C_M)\geq 1.$ Now, from Corollary \ref{ros-extn}, the desired result follows.
	
	To prove the converse, we first assume that $k\subset k\langle w\rangle\subseteq k\langle y_1\rangle,$ where $w$ is a nonalgebraic solution of a Riccati differential equation corresponding with a monic linear operator $\mathcal L\in k[\partial]$ of order $n.$ Let $E$ be a Picard-Vessiot extension of $k$ for $\mathcal L$ and $l:=\mathrm{tr.deg}(E|k).$  Let $y_1, y_2,\dots, y_{m}\in \mathfrak S_{k,f,y_1}.$ We shall show that $m\leq l$ and this will prove $\mathfrak S_{k,f,y_1}$ is a finite set. Let $M:=k\langle y_1,\dots, y_m\rangle$ and   $L$ be a Picard-Vessiot extension of $M$ for $\mathcal L\in k[\partial]\subseteq M[\partial]$ and $V\subset L$ be the set of all solutions of $\mathcal L(t)=0$ in $L.$ Then since $C_M=C,$ $k\langle V\rangle$ is a Picard-Vessiot extension of $k$ for $\mathcal L\in k[\partial].$ Therefore, $E$ and $k\langle V\rangle$ are $k-$isomorphic differential fields.  
    
    Let $w_1:=\zeta'/\zeta,$ $p:=\mathrm{tr.deg}(k\langle w_1\rangle|k)$ and for each $i=2, 3,\dots, m,$ let $w_i=\sigma_i(\zeta'/\zeta).$ Let $W$ be the set of all solutions in $L$ of the  Riccati differential equation $\mathcal R_\mathcal L(w, w',\dots,w^{(n-1)})=0.$ It is easily seen that the map $l\delta: V\setminus \{0\}\to W,$ $l\delta(x)=x'/x,$ is surjective.  This, in particular, implies $w_1,\dots, w_m\in W\subset k\langle V\rangle$ and we obtain that $$m\leq mp=\mathrm{tr.deg}(k\langle w_1,\dots, w_m\rangle)\leq l=\mathrm{tr.deg}(k\langle V\rangle|k).$$ Thus, $\mathfrak S_{k,f,y_1}$ is a finite set. 

Finally, we assume that $k\subset K\subseteq k\langle y_1\rangle$ and that $K$ is an abelian extension of an algebraic extension $\tilde{k}$ of $k.$ Note that $\mathfrak U$ is also universal over $\tilde{k}.$ From Proposition \ref{algindp-auto}, we know that $\mathfrak S_{\tilde{k},f,y_1}=\{y_1\}.$ Thus if $y_2\in \mathfrak S_{k,f,y_1}$ then since \begin{align*}\mathrm{tr.deg}(k\langle y_1,y_2\rangle|k)&=\mathrm{tr.deg}(\tilde{k}\langle y_1,y_2\rangle|\tilde{k})\\ &<\mathrm{tr.deg}(\tilde{k}\langle y_1\rangle|\tilde{k})+\mathrm{tr.deg}(\tilde{k}\langle y_2\rangle|\tilde{k})\\ &=\mathrm{tr.deg}(k\langle y_1\rangle|k)+\mathrm{tr.deg}(k\langle y_2\rangle|k),\end{align*}
we must have $y_1=y_2$ and we obtain that $\mathfrak S_{k,f,y_1}=\{y_1\}.$
\end{proof}

\begin{proof}[Proof of Theorem \ref{algebraicdependence-aode}(\ref{auto-nongeneral-SN})]
From Theorem \ref{algebraicdependence-aode}(\ref{nongeneral-SN}), we obtain a $C-$irreducible differential subfield $M$ of  $C\langle y_1\rangle$  that can be embedded in a strongly normal extension of $C$ and from  Remark \ref{remarksonintromainSN}, we conclude that $M$ has the desired structure. 
Converse readily follows from Proposition \ref{algindp-auto} and from the definition of $\mathfrak S_{k,f,y_1}$. 
\end{proof}

\begin{proof}[Proof of Theorem \ref{algebraicdependence-aode}(\ref{cardn+2-order4})] We only need to show that $\mathfrak S_{k,f,t}$ is finite  implies $\mathfrak S_{k,f,t}$ has at most $n+2$ elements. Suppose that $\mathfrak S_{k,f,y_1}=\{y_1,\dots, y_m\}.$  From Theorem \ref{algebraicdependence-aode}(\ref{nongeneral-SN}), it is enough to consider the case when there are a monic $\mathcal L\in k[\partial]$ of order $r$ and an element $w\in k\langle y_1\rangle$ such that $\mathcal R_\mathcal L(w,w',\dots, w^{(r-1)})=0$ and  that $K:=k\langle w\rangle$ is  $k-$irreducible. For $i=2, \dots, m,$ we also have solutions $w_i\in k\langle y_i\rangle$ of $\mathcal R_\mathcal L(w,w',\dots, w^{(r-1)})=0.$  Let $L$ be  a Picard-Vessiot extension of $k$ for $\mathcal L$ with $w_1,\dots, w_m\in L$ (see the proof of Theorem \ref{algebraicdependence-aode}(\ref{nongeneral-SN})).  Let $l:=\mathrm{ tr. deg}(K|k).$ Then, $$l\leq n\ \ \text{and} \ \ lm=\mathrm{tr.deg}(k\langle  w_1,\dots, w_m \rangle|k)\leq  \mathrm{tr.deg}(L|k).$$

Thus, it is enough to show that $\mathrm{tr.deg}(L|k)\leq l(l+2).$ Let $\mathscr G:=\mathscr G(L|k).$ We shall split the proof into a few cases, as done in the proof of Theorem \ref{intromain}.

Suppose  that $L^\mathscr U\cap K=k,$ where $\mathscr U$ is the unipotent radical of $\mathscr G.$ Then, from Lemma \ref{unipotent-case}, we may further assume that $\mathcal L=\mathcal L_1\mathcal D,$ where $\mathrm{ord}(\mathcal L_1)=1$ and $\mathrm{ord}(\mathcal L)\leq l+1.$ Since $\mathcal L$ has a right factor,   $\mathscr G$ is a proper (connected) subgroup of $\mathrm{GL}_{l+1}(C)$ and we obtain that $\mathrm{tr.deg}(L|k)\leq (l+1)^2-1=l(l+2).$ 

Suppose that $L^\mathscr R\cap K=k$ and that $\mathscr R$ is the radical of $\mathscr G,$ then from Lemma \ref{torus-case}, we obtain that $K$ contains the Picard-Vessiot extension $k(\zeta)$ of $k$ where $\zeta\in K\setminus k$ and $\zeta'/\zeta\in k.$ Then, by Proposition \ref{algindp-auto}, we in fact obtain $m\leq 2.$ 

Suppose that $\mathscr G$ is a semisimple algebraic group of rank $r$. Let $K^0$ be the algebraic closure of $K$ in $E$ and choose a differential subfield $M$ of $K^0$ so the $M^0$ is $k-$irreducible. Let $\mathscr H:=\mathscr G(L|M^0).$ Then, as observed in the proof of Theorem \ref{semisimple-case}, $\mathscr H$ is either  semisimple or we may assume $\mathscr G$ is simple and $\mathscr H$ is parabolic.  We also know that $$\mathrm{tr.deg}(M^0|k)\leq l\leq \mathrm{tr.deg}(k\langle y_1\rangle|k)\leq 4$$ and that   $\mathrm{rank}(\mathscr G)\leq l\leq 4.$

If $\mathscr H$  is semisimple then $\mathscr G/\mathscr H$ is an affine homogeneous space and therefore, from \cite[Section 5]{Kaga-Watabe} and \cite{Pop75},  $\mathscr G$ must be isogenus to one of the following groups: $$\mathrm{SL}_3(C), \quad \mathrm{SL}_2(C)\times \mathrm{SL}_2(C), \quad\mathrm{SL}_2(C).$$ In particular, all the simple component of $\mathscr G$ are classical simple algebraic groups. Therefore, from Remark \ref{bound-r(r+2)},  we obtain $$\mathrm{tr.deg}(L|k)=\mathrm{dim}(\mathscr G)\leq r(r+2)\leq l(l+2).$$ 

Finally, if $\mathscr H$ is parabolic and $\mathscr G$ is simple then since none of the exceptional simple algebraic groups have subgroups of codimension $\leq 4$, we obtain that $\mathscr G$ must be a classical simple group. Thus,  $\mathrm{tr.deg}(L|k)=\mathrm{dim}(\mathscr G)\leq r(r+2)\leq l(l+2).$ \end{proof}

In the next two examples, we consider the Lotka-Volterra system and the Poizat differential equation and prove some results in view of Theorem \ref{algebraicdependence-aode}.  
\begin{example}
   Consider  the Lotka-Volterra system  \begin{equation}\tag{L-V}
\left\{ \begin{aligned} 
	t_1' &= \alpha t_1+\beta t_1t_2\\
	t_2' &=\gamma t_2+\delta t_1t_2\\
\end{aligned} \right.
\end{equation} where $\alpha,\beta,\gamma,\delta\in C\setminus 0.$ From \cite[Theorem 4.2]{eagles-Jimenez}, we know that the system has a generic solution $y:=(y_1, y_2)\in \mathfrak U^2$ and the differential field $C(y_1,y_2)$ contains no elements $x$ such that $x'=1$ or $x'=cx$ for some nonzero $c\in C$ if and only if $\alpha\neq \gamma.$ Thus, if $\alpha\neq \gamma,$ then since $C(y_1,y_2)$ is a rational field, we obtain from Theorems \ref{intromain-SN} and \ref{algebraicdependence-aode} that $\mathfrak S_{C,\text{L-V},y}$ is infinite. 
\end{example}

\begin{example}(\cite{FJMN23})\label{example-poizat}
Consider the Poizat differential equation  \begin{equation}\label{poizateqn}f(t,t',t''):=t''-t'h(t)=0,\end{equation} where $h$ is a nonzero rational function in one variable over $C$.  We prove the following statements
\begin{enumerate}[(A)]
    \item \label{poizat-nnc} The equation (\ref{poizateqn}) has no generic solution  if and only if $h= \partial g/\partial t$ for some $g\in C(t).$ \\
    \item \label{poizatgeneric-algindp} If $y$ is a generic  solution of (\ref{poizateqn}) then $\mathfrak S_{C,f,y}$ is infinite. \\
    \item \label{trdonesolns-poizat} Suppose that $h(t)= \partial g(t)/\partial t$ for some  $g\in C(t).$ Then any solution $y$ of the equation (\ref{poizateqn}) is a solution of  $g_{c}(t,t'):=t'-g(t)-c=0$ for some $c$ in $C.$ Moreover, for any nonconstant solution $y$ of $g_{c}(t,t')=0,$ the set $\mathfrak S_{C,g_{c},y}=\mathfrak S_{C,f,y}=\{y\}$  if and only if either \begin{equation}\label{gdescription} \frac{1}{g-c}=\frac{\partial u}{\partial t} \ \ \text{or}\ \frac{1}{g-c}=\frac{1}{du}\frac{\partial u}{\partial t}\end{equation} for some $d\in C, u\in C(t).$
\end{enumerate}

\textsl{Proof of (\ref{poizat-nnc}).} Suppose that $h=\partial g/\partial t$ for some rational function $g\in C(t)$ and $y$ be a  solution of (\ref{poizateqn}). Then $(y'-g(y))'=0.$ Since $C_{C\langle y\rangle}=C,$ we must have for some $c\in C$ that $y'=g(y)+c\in C(y).$ That is, $C\langle y\rangle=C(y)$ and we obtain that $y$ is not a generic solution. 

To prove the converse, suppose that every solution of (\ref{poizateqn}) is nongeneric and consider the rational function field $C(x_1,x_2)$ with the derivation $$x_2\frac{\partial}{\partial x_1}+x_2h(x_1)\frac{\partial}{\partial x_2}.$$
We may assume that this differential field is contained in $\mathfrak U.$ Since $x''_1=x'_1h(x_1),$  there must be a constant in $C(x_1,x_2)\setminus C.$ Let $P, Q\in C(x_1)[x_2]$ be relatively prime polynomials, $Q$ be monic and $(P/Q)'=0.$ Then,
 $$Q\left(\frac{\partial P}{\partial x_1}+h(x_1)\frac{\partial P}{\partial x_2}\right)=P\left(\frac{\partial Q}{\partial x_1}+h(x_1)\frac{\partial Q}{\partial x_2}\right).$$
Therefore, there is an element $z\in C(x_1)[x_2]$ such that $$\left(\frac{\partial Q}{\partial x_1}+h(x_1)\frac{\partial Q}{\partial x_2}\right)=z Q\quad \text{and}\quad \left(\frac{\partial P}{\partial x_1}+h(x_1)\frac{\partial P}{\partial x_2}\right)=zP.$$

Writing $P=a_nx^n_2+a_{n-1}x^{n-1}_2+\cdots+a_0\in C(x_1)[x_2],$  $a_n\neq 0, a_{n-1},\dots, a_0\in C(x_1)$ and $Q=x^m_2+b_{m-1}x^{m-1}_2+\cdots+b_0\in C(x_1)[x_2], b_{m-1},\dots, b_0\in C(x_1),$ we obtain that \begin{align}& \frac{\partial a_n}{\partial x_1}x^n_2+\frac{\partial a_{n-1}}{\partial x_1}x^{n-1}_2+\cdots+\frac{\partial a_0}{\partial x_1}+na_nh(x_1)x^{n-1}_2+\dots+h(x_1)a_1 =z(a_nx^n_2+\dots+a_0)\label{equationforP}\\
&\frac{\partial b_{m-1}}{\partial x_1}x^{m-1}_2+\cdots+\frac{\partial b_0}{\partial x_1}+mh(x_1)x^{m-1}_2+\dots+h(x_1)b_1=z(x^m_2+b_{m-1}x^{m-1}_2+\dots+b_0)\label{equationforQ}
\end{align}
From Equation (\ref{equationforQ}), it is clear that $z=0.$ Now from Equation (\ref{equationforP}), we obtain that $a_n\in C$ and that $$h(x_1)=\frac{\partial g}{\partial x_1}, \ \text{where } g=\frac{-a_{n-1}}{na_n}.$$

\textsl{Proof of (\ref{poizatgeneric-algindp})}. If $y$ is a generic solution of the Poizat equation then it can be shown that there is no element $z\in C(y,y')$ such that $z'=1$ or $z'=cz$ for some nonzero $c\in C$ (\cite{eagles-Jimenez}). Therefore, from Theorem \ref{algebraicdependence-aode}, we obtain that $\mathfrak S_{C,f,y}$ is infinite.

\textsl{Proof of (\ref{trdonesolns-poizat})}. Observe that $y''=g(y)'=y'h(y).$ Therefore $y'-g(y)=c\in C.$ Since $C(y)$ is a rational field, by Theorems \ref{intromain-SN} (\ref{intromain-uni}) and \ref{algebraicdependence-aode} (\ref{auto-nongeneral-SN}), $\mathfrak S_{C,f,y}=\{y\}$ if and only if there is an element $u\in C(y)\setminus C$ such that $u'=1$ or $u'=du$ for some nonzero $d\in C$ if and only if  $u\in C(y)$  satisfying (\ref{gdescription}).
\end{example}

\begin{remark}
 Let $f(X,Y,Z)=Z-nYX^{n-1}$ for some $n\geq 3.$ Consider the Poizat equation $f(t,t',t'')=t''-nt't^{n-1}=0.$ Let $y_1, y_2\in \mathcal U$ be nonconstant solutions of the equations $t_1'-t_1^n=0$ and $t_2'-t_2^n+1=0$ respectively. Then, both $y_1$ and $y_2$ are solutions of the Poizat equation.  Observe that $1/t^n=\partial u(t)/\partial t$ for $u(t)=1/((n+1)t^{n+1}).$ However, there are no rational function $u$ and a nonzero constant $c$ such that $$\frac{1}{t^n-1}= 
     \frac{\partial u(t)}{\partial t} \ \text{or}\ 
    \frac{\partial u(t)/\partial t}{cu(t)}$$
 This shows that $\mathfrak S_{C,f,y_1}=\{y_1\},$ whereas, $\mathfrak S_{C,f,y_2}$ is an infinite set. 
\end{remark}

\section{Applications to $d-$solvability of linear differential equations}
In this section, we prove Theorem \ref{sol-td-higher}. We start with an equivalent definition of $d-$solvable extensions. A differential field  extension $E\supset k$ is said to be \emph{$d-$solvable} if there exists a tower of differential fields $$k=K_0 \subseteq\cdots\subseteq K_r$$ such that $E \subseteq K_r$ and for each $i, 1\leq i\leq r,$ either \begin{enumerate}[(i)]
    \item  $K_{i}$ is a finite algebraic extension of $K_{i-1},$ or
    \item $K_{i}$ is a Picard-Vessiot extension of $K_{i-1}$ such that $\mathscr G(K_{i}|K_{i-1})$ is a subgroup of $\mathrm{GL}_d.$
\end{enumerate}
 A linear algebraic group $\mathscr G$ is said to be \emph{$d-$solvable} if there exists a tower of closed subgroups $$\mathscr G_0=\{e\} \subseteq \mathscr G_1\subseteq \cdots \subseteq \mathscr G_r=\mathscr{G}$$ such that each $\mathscr G_{i-1}$ is normal in $\mathscr G_{i}$ and $\mathscr G_{i}/\mathscr G_{i-1}$ is either finite or a subquotient of $\mathrm{GL}_d.$ 
 
 It is known that a Picard–Vessiot extension $E\supset k$ is $d-$solvable if and only if the differential Galois group $\mathscr G (E|k)$ is $d-$solvable (see \cite[Theorem 1]{Ngu09}). We will also be using this equivalent definition to determine the  $d-$solvability of certain Picard-Vessiot extension. It is also known that  a $d^+-$solvable extension is a $d-$solvable extension for $d\geq2$ and that a $1^+-$solvable extension is not $1-$solvable, but it is $2-$solvable (see \cite[page 427]{Ngu09}). 

 The following two lemmas will be used in proving Theorem \ref{sol-td-higher}.
 
\begin{lemma}(cf. \cite[Theorem 3.1]{URVRS-21})\label{no-pv-element}
    Let $E$ be a Picard-Vessiot extension of $k$ and $k^0$ be the algebraic closure of $k$ in $E$. Suppose that  $\mathscr G:=\mathscr{G}(E|k)$ contains a proper parabolic subgroup  $\mathscr{P}$.  Then $T(E^\mathscr{P}|k)=k^0.$ 
\end{lemma}
\begin{proof}
    Let $\mathscr G^0$ be the identity component of $\mathscr G$.  Then $k^0=E^{\mathscr G^0}$ and $E$ is a Picard-Vessiot extension of $k^0$ with $\mathscr G(E|k^0)=\mathscr G^0.$ Note that $T(E^\mathscr P|k)=T(E^\mathscr P|k)$ and therefore $k^0\subset T(E|k)^\mathscr P.$ Now let $y\in T(E|k^0)^\mathscr P,$ then the orbit set $\mathcal{O}_y$ of $y$ under the action of $\mathscr G^0$ is contained in a finite dimensional $\mathscr G^0-$stable $C-$vector space.  The induced map $\phi: \mathscr G^0/\mathscr P \to \mathcal{O}_y$, given by $\phi(\sigma+\mathscr P)=\sigma(y)$ for $\sigma \in \mathscr G^0$, is a morphism from a projective variety into some affine space containing $\mathcal{O}_y$. Thus $\phi$ must be a constant. That is, $y \in T(E|k^0)^{\mathscr G^0} = k^0.$ Since $T(E|k^0)=T(E|k),$ we now obtain that $k^0 = T(E^\mathscr P|k)$.  
\end{proof}

\begin{lemma}\label{solvability-pv-subfield}
    Let $E$ be a Picard-Vessiot extension of $k$ with a connected Galois group $\mathscr G:=\mathscr G(E|k)$. Let $K$ be a $k-$irreducible differential subfield of $E$ such that $K$ is algebraically closed in $E$ and $\mathrm{tr. deg}(K|k)=d\leq 4.$ Then $K$ is $(d+1)-$solvable extension of $k.$ Furthermore, if $T(K|k)\neq k,$ then $K$ is $d^+-$solvable for $d=1$ and $d-$solvable for $2\leq d\leq 4.$   
\end{lemma}

\begin{proof}
    We first observe that one of the four situations mentioned in Remark \ref{obv-intromain} must hold for the differential field $K$ and accordingly we divide the proof of the lemma in to these four cases. 

    Suppose that  Remark \ref{obv-intromain} (\ref{from-unipotent}) holds. Then, $K$ is $d^+-$solvable extension of $k.$ Since a $d^+-$solvable extension is a $d-$solvable extension for $d\geq2$, it follows that $K$ is $d^+-$solvable for $d=1$ and $d-$solvable for $2\leq d\leq 4.$  
    
    Suppose that Remark \ref{obv-intromain} (\ref{from-torus}) holds then it immediately follows that $K$ is $1-$solvable. 

    Next, assume that Remark \ref{obv-intromain} (\ref{from-semisimple}) holds.  Then, $\mathscr {G}_1/\mathscr H$ is an affine homogeneous space, where $\mathscr G_1=\mathscr (F|k)$. Since $\mathrm{rank}(\mathscr G_1)\leq d-1,$ we immediately see that $d\neq 1,2.$  From \cite[section 3]{Pop75} it follows that $d\neq 3.$ Thus we are left to deal with the case when $d=4.$  In this case, from \cite[section 5]{Kaga-Watabe}, the only possible choices for $\mathscr G_1$ (upto isogeny) are   $\mathrm{SL_3}$ or $\mathrm{SL}_2\times \mathrm{SL}_2.$ Since $K$ is contained in $F,$ from \cite{Ngu09}[Theorem 2 \& table 1], $\mathscr G_1$ is $3-$solvable. Thus, we obtain that $K$ is $3$-solvable.

    Finally, suppose that Remark \ref{obv-intromain}(\ref{from-parabolic}) holds and let $\mathscr{G}_2:=\mathscr{G}(L|k).$ Maximal parabolic subgroups of simple algebraic groups having minimal codimension are known (see \cite[Table 1]{GS91}). This allows us to conclude the following 
    
    \begin{enumerate}[(i)]
        \item If $d=1$ then  $\mathscr G_2$ is isogeneous to $\mathrm{SL}_2,$ which is $2-$solvable. 
        \item If $d=2$ then  $\mathscr G_2$ is isogeneous to $\mathrm{SL_3},$ which is $3-$solvable. 
        \item If $d=3$ then $\mathscr G_2$ is isogeneous to either $\mathrm{SL_3}$ or $\mathrm{SL}_4$ or $\mathrm{SP}_4$. Since all these groups are subgroups of $GL_4,$ they are all $4-$solvable. 
        \item If $d=4$ then $\mathscr G_2$ is isogeneous to either $\mathrm{SL_3}$ or $\mathrm{SL}_4$ or $\mathrm{SP}_4$ or $\mathrm{SL}_5.$ In all these cases $\mathscr G_2$ is $5-$solvable.
    \end{enumerate}        

From the above conclusions and from \cite[Theorem 1]{Ngu09}, $L$ is $d+1-$ solvable and consequently we obtain $K$ is $d+1-$solvable.

To prove the last assertion of the lemma, assume that $T(K|k)\neq k.$ Note that $\mathscr G$ is connected and therefore $k$ is algebraically closed in $E.$ By Lemma \ref{no-pv-element}, $K$ cannot be the fixed field of a parabolic subgroup of the Galois group. Therefore, Remark \ref{obv-intromain}(\ref{from-parabolic}) can not hold for $K.$ This proves that $K$ is $d^+-$solvable for $d=1$ and $d-$solvable for $2\leq d\leq 4.$  \end{proof}

\begin{proof}[Proof of Theorem \ref{sol-td-higher}]
 Let $E$ be a Picard-Vessiot extension of $k$ and $K$ be the algebraic closure of $k\langle y\rangle$ in $E.$ Observe that $y$ is  $d-$solvable if and only if $K$ is a $d-$solvable extension of $k.$

 If $K$ is $k-$irreducible, then the first part of the theorem follows from Lemma \ref{solvability-pv-subfield}. Therefore, we can assume that $\mathrm{tr. deg}(K|k)\geq 2$ and $K$ is not $k-$irreducible. Then, $n\geq 2$ and $K$ can be resolved into a  tower of differential fields
  $k=K_0\subset K_1\subset\cdots\subset K_n=K,$ where for each $i=1,2,\dots, n,$ $K_i$ is $K_{i-1}-$irreducible differential field and that $K_i$ is algebraically closed in $E.$ Since for each $i,$ $\mathscr G(E|K_{i-1})$ is connected, by Lemma \ref{solvability-pv-subfield}, we obtain that $K_{i}$  is $d-$solvable extension of $K_{i-1}$.  This proves that $K$ is  a $d-$solvable extension of $k.$

 To prove the second part of the theorem, let $\mathcal L$ be irreducible.  The set of all $d-$solvable solutions $V$ of $\mathcal L(t)=0$ forms a $\mathscr G(E|k)-$stable $C-$vector space. Now from irreducibility of $\mathcal L,$ it follows that if $V\neq 0$ then $\mathrm{dim}_C(V)=\mathrm{ord}(\mathcal L)$ (\cite[Corollary 2.35]{vdPS03}). Thus, in particular, we observe that if $\mathcal L(t)=0$ has a $d-$solvable solution for some $d$ then it has a $C-$basis of $d-$solvable solutions. 

 Let $2\leq d\leq 4.$ Then,
 from the above observation and from the first part of the theorem, it readily follows that $\mathcal L(t)=0$ has a $C-$basis of $d-$solvable solutions. Now 
 assume that $d=1.$  Then $k\langle y\rangle$ is $k-$irreducible and therefore from the first part of the theorem, we obtain that $y$ is a $1^+-$solvable solution of  $\mathcal{L}(t)=0$. Since $1^+-$solvable extension fields are subfields of liouvillian extension fields,  $\mathcal L(t)=0$ has a nonzero liouvillian solution $z\in E$ such that $z'/z$ is algebraic over $k$ (\cite[Proposition 1.45]{vdPS03}). Clearly $z$ is a $1-$solvable solution of $\mathcal{L}(t)=0.$ Now, from the observation, we obtain that $\mathcal L(t)=0$ has a basis of $1-$solvable solutions. \end{proof}

\begin{remark}
   The second part of the Theorem \ref{sol-td-higher}  can be rephrased as follows: If $\mathcal L\in k[\partial]$ is irreducible and $\mathcal{L}(t)=0$ has a solution $y$ such that $d:=\mathrm{tr.deg}(k\langle y\rangle \leq 4$, then the differential Galois group of $\mathcal{L}(t)=0$ is $d-$solvable.
\end{remark}



\bibliographystyle{abbrv}
\bibliography{PKVRS-DESNE}
\end{document}